\documentclass{my_m2an} 

\usepackage{loc}
\usepackage{lop}

\newcounter{cst}
\newcommand{\ctel}[1]{C_{\refstepcounter{cst}\label{#1}\thecst}}
\newcommand{\cter}[1]{C_{\ref{#1}}}

\title{A staggered scheme for the compressible Euler equations on general 3D meshes}
\date{}
\begin{document}

\author{Aubin Brunel}
\address{Aix-Marseille University, CNRS, France (aubin.brunel@univ-amu.fr)}
\author{Rapha\`ele Herbin}
\address{Aix-Marseille University, CNRS, France (raphaele.herbin@univ-amu.fr)}
\author{Jean-Claude Latch\'e}
\address{Institut de S\^uret\'e et de Radioprotection Nucl\'eaire (IRSN), France (jean-claude.latche]@irsn.fr)}

\begin{abstract}
    We address here the discretization of the momentum convection operator for fluid flow simulations on 2D triangular and quadrangular meshes and 3D polyhedral meshes containing hexahedra, tetrahedra, prisms and pyramids.
    The finite volume scheme that we use for the full Euler equations is based on a staggered discretization: the density unknowns are associated with a primal mesh, whereas the velocity unknowns are associated with a "fictive" dual mesh.
    Accordingly, the convection operator of the mass balance equation is derived on the primal mesh, while the the convection operator of the momentum balance equation is discretized on the dual mesh.
    To avoid any hazardous interpolation of the unknowns on  a possibly ill-defined dual mesh, the mass fluxes of the momentum convection operator are computed from the mass fluxes of the mass balance equation, so as to ensure the stability of the resulting operator.
    A coherent reconstruction of these dual fluxes is possible, based only on the kind of considered polygonal or polyhedral cell, and not on each cell itself.
    Moreover, we show that this process still yields a consistent convection operator in the Lax-Wendroff sense, that is, if a sequence of piecewise constant functions is supposed to converge to a a given limit, then the weak form of the corresponding discrete convection operator converges to the weak form of the continuous operator applied to this limit.
    The derived discrete convection operator applies to  both constant and variable density flows and may thus be implemented in a scheme for incompressible or compressible flows.
    Numerical tests are performed for the Euler equations on several types of mesh, including hybrid meshes, and show the excellent performance of the method.
\end{abstract}


\maketitle

%

We address in this paper a numerical scheme for the Euler equations, which read:
\begin{subequations}\label{eu:eq:pb}
    \begin{align}
    \label{eq:cont_mass_balance} &
    \partial_t \rho + \dive( \rho\, \bfu) = 0,
    \\[1ex] 
    \label{eu:eq:pb_mom} &
    \partial_t (\rho\, u_i) + \dive(\rho\, u_i \, \bfu) + \partial_i p= 0, \quad 1 \leq i \leq d,
    \\[1ex] 
    \label{eu:eq:pb_Etot} &
    \partial_t (\rho\, E) + \dive(\rho \, E \, \bfu) + \dive ( p \, \bfu )=0,
    \\ 
    \label{eu:eq:pb_etat} &
     p=(\gamma-1)\, \rho\, e, \qquad E=\frac 1 2|\bfu|^2+e,
    \end{align} 
\end{subequations}
where the quantities $\bfu$, $\rho$, $p$, $E$ and $e$ refer respectively to the velocity of the fluid, its density, its pressure, the total energy and the internal energy, and $\gamma$ is a coefficient specific to the fluid and is supposed strictly greater than $1$.
The problem is supposed to be posed over a spatial domain $\Omega$ and on a time interval $[0,T]$, where $\Omega$ is an open bounded connected subset of $\xR^d$ with $1 \leq d \leq 3$.
The system is supplemented by the following initial conditions:
\begin{equation}\label{eqdef:ci}
    \rho(\bfx,0) = \rho_0(\bfx), \;  \bfu(\bfx,0) = \bfu_0(\bfx), \mbox{ with } \rho_0 \in L^\infty(\Omega),  \bfu_0 \in L^\infty(\Omega)^d,
\end{equation}
and suitable boundary conditions.

In a previous paper  \cite{gas-18-mus} co-signed by two of the above authors, a staggered in space and segregated in time scheme, involving only explicit steps, was developed to approximate the solutions of the Euler equations on general simplicial or quadrilateral/hexahedral meshes.
Our aim in the present work is twofold: first we generalize the scheme to prismatic and pyramidal meshes and give a precise definition of the convective fluxes for these cells; second, we give a proof of the Lax-Wendroff consistency of the scheme for general meshes, in the sense that if a sequence of approximate solutions is assumed to converge to a certain limit as the time and space steps tend to zero, then this limit is necessarily a weak solution of the Euler equations.

Let us first recall the essential features of the proposed numerical scheme.
\begin{itemize}
    \item  The scheme features a staggered arrangement of the unknowns: the scalar variables (density, pressure) are approximated by piecewise constant functions over the cells while the velocity is approximated at the faces of the cells.

    \item The equations are discretized by a finite volume scheme. 
 
    \item Although not fully explicit, the scheme is segregated and each step is explicit, in the sense that the balance equations are solved successively and do not require any linear system solver.

    \item  The scheme solves the internal energy balance rather than the total energy balance, and thus conserves the positivity of the internal energy;
    a corrective term is added in this equation in order to avoid wrong shock solutions.
    
    \item Upwinding or more precise procedures are avalaible for the convection fluxes, which ensures the positivity of the density, internal energy and pressure.
\end{itemize}

The scheme is related to the family of flux splitting schemes of the references \cite{ste-81-flux,lio-93-new,zha-93-num,lio-06-seq,tor-12-flux}. 
However it differs from them because of the use of staggered discretization which has the advantage of providing a discrete inf-sup stability condition, and by the use of the internal energy equation. 
Moreover, the pressure gradient is discretized as the dual of the velocity divergence, which, coupled with the preservation of the positivity of the density, yields a discrete analogue of the conservation of the total energy.
The continuous mass balance equation \eqref{eq:cont_mass_balance}
is discretized on a primal mesh, whereas the momentum convection operator for the $i$-th component of the velocity (with $1 \leq i \leq d$, where $d$ is the dimension of the problem), given by
\begin{equation} \label{eq:cont_conv}
    (\mathcal{C} u)_i := \partial_t(\rho \, u_i)+\dive(\rho \, u_i \, \bfu),
\end{equation} 
is discretized on the dual mesh.
As a consequence, the mass fluxes are first defined on the primal mesh and the densities and mass fluxes are reconstructed on the dual mesh from the primal mesh densities and mass fluxes respectively, so as to ensure that a mass balance holds on the dual cells; this is a crucial step to ensure the $L^2$-stability of the scheme, see e.g. \cite{gal-08-unc}.

\medskip 

The first goal of this paper is to propose a general reconstruction method of these dual cell mass fluxes (and thus, of the momentum convection operator) on a wide range of cells, including prismatic and pyramidal cells in 3D, which are often encountered in industrial meshes. 
Indeed, industrial 3D meshes are often based on both hexahedral and tetrahedral elements cells, and the usual way to connect a tetraheda to a heaxhedra is to use pyramids
or prisms \cite{mav-97-uni,bis-98-tet,owe-00-for,sor-2003-mul,bau-14-fron}.
The scientific software CALIF3S \cite{califs} that is routinely used at IRSN for nuclear safety studies supports the use of  hybrid tetrahedral-hexahedral-prismatic-pyramidal meshes.

In the present work, we generalize the reconstruction of the dual fluxes from the primal fluxes which was introduced for simplicial and hexahedral meshes (see e.g. \cite{gas-18-mus,her-14-som}) to these hybrid meshes, and the requirement of a local mass balance equation leads to the problem of the resolution linear system.
The resulting system that needs to be solved for this reconstruction  may be underdetermined;  a procedure is constructed in order to find a solution of the system for a given polygon or polyhedron, regardless of its possible distorsion or anisotropy.

\medskip

The construction of the dual fluxes is thus based on a set of algebraic equations, rather than on integration on some well defined control volume; this might seem unorthodox in the framework of the finite volume method.
However, to support the validity of this method, we give a general consistency result: this is the second goal of the paper.
Indeed, we show that the derived discrete convection operator is weakly consistent in the Lax-Wendroff sense, \ie~ that its weak form tends to the weak form of the continuous convection operator as the  space and time steps both tend to 0.
 
\medskip
 
The paper is organized as follows. 
In section \ref{sec:mesh}, we give the time and space discretizations.
The mesh definition is given such as to take into account several kinds of grid cells in two or three space dimensions. 
Section \ref{sec:scheme} is devoted to the definition of the scheme, with special emphasis on the construction of the convection operator and its link with the mass balance equation, which is the subject of Section \ref{sec:vect_conv_op}.
In the case of pyramidal and prismatic cells, the linear systems satisfied by the dual fluxes are underdetermined; we give the  specific choice of a solution that is implemented in the code. 
The scheme is shown to be consistent in the Lax-Wendroff sense in section \ref{sec:consis}.
The behaviour of the numerical scheme on some  test cases is investigated in section \ref{sec:num}.
Finally, some lemmas needed for the Lax-Wendroff consistency are recalled in the appendix.  
 
\section{Space and time discretization} \label{sec:mesh}

Due to the staggered nature of the scheme, the discretization consists in a primal mesh and a dual mesh that is derived from the primal one.
We therefore first define the primal mesh $\mesh$, which is obtained by splitting the domain $\Omega$ into a finite family of disjoint polygons (triangles and quadrangles) when $d=2$ and polyhedra (hexahedra, tetrahedra, prisms and pyramids) when $d=3$.
Following, we will refer to a primal element as a control volume or a cell.
The set of the faces (if $d=3$, or edges if $d=2$)  of the mesh is denoted by $\edges$.
It is split into $\edges = \edgesint \cup \edgesext$, where $\edgesext$ is the set of external faces (or edges) and $\edgesint$ is the set of internal faces (or edges), that we define as follows. 
A face (or edge) $\edge \in \edges$ is said to be an external face (or edge), and thus $\edge \in \edgesext$, if it is part of the boundary of the domain (\ie\ $\edge \subset \partial \Omega$), whereas $\edge \in \edges$ is said to be an internal face (or edge), and thus $\edge \in \edgesint$, if there exists $(K,L) \in \mesh^2$ with $K \neq L$ such that $\overline K \cap \overline L = \edge$.
Such a face (or edge) is thus denoted $\edge = K | L$.
Moreover, for $K \in \mesh$ and $\edge \in \edges$, we denote by $|K|$ the measure of $K$ and $|\edge|$ the $(d-1)$-measure of $\edge$. 

Then, the dual mesh is constructed as follows: for a regular polygon or a regular polyhedron $K \in \mesh$, we denote by $\bfx_K$ the mass center of $K$ and we construct $D_{K,\edge}$ as the cone with basis $\edge$ and with vertex $\bfx_K$; this definition is extended to a general cell $K$, by supposing that $K$ is split in the same number of sub-cells (the geometry of which does not need to be specified) and with the same connectivity. 
We can now define for each $\edge \in \edges$ its associated dual cells $D_\edge$.
When $\edge \in \edgesint$ with $\edge = K|L$, we define $D_\edge := D_{K,\edge} \cup D_{L,\edge}$; if $\edge \in \edgesext \cap \edges(K)$, then we have $D_{\edge} := D_{K,\edge}$.
We then denote by $\edgesd(D_\edge)$ the set of dual faces (or edges) of $D_\edge$, and by $\edged=D_\edge|D_{\edge'}$ the face (or edge) separating two dual cells $D_\edge$ and $D_{\edge'}$. 

Next, we associate the unknowns associated with the scalar variable, such as the pressure and the density for instance, the cells of the primal mesh $\mesh$, and they are denoted in those instances by $p_K$ and $\rho_K$ for a cell $K \in \mesh$.
On the other hand, the degrees of freedom for the velocity are linked to the dual mesh faces and are denoted $\bfu_\edge= (u_{\edge,1},\ldots,u_{\edge,d})$ for an edge $\edge \in \edges$.
All the components of the velocity are thus approximated on each face of the mesh, and their degrees of freedom are identified to the mean value of the velocity component over the face. 

An example of the discretization with a few control volumes is given on Figure \ref{fig:mesh}.

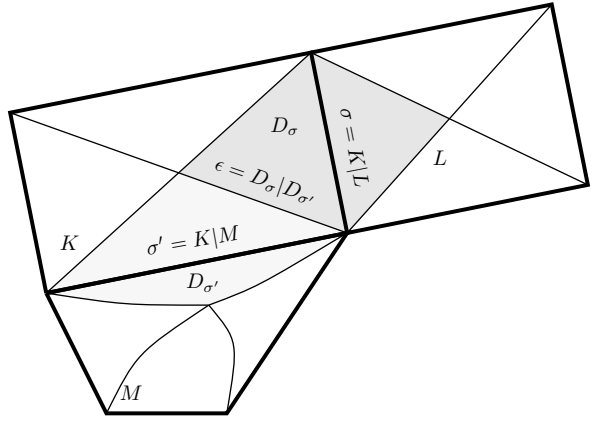
\begin{figure}
\begin{center}
\begin{tikzpicture}[scale = 0.8, every node/.style={scale=0.8}]
   \fill[color=white!90!black](3.2,3) -- (6,2) -- (7.7,3.9) -- (5.4,5) ;
   \draw (4.6,3.5) node[anchor = south west] {{$D_\edge$}} ;
   \fill [color=white!97!black] (3.2,3) -- (6,2) .. controls (4.5, 1.1) .. (3.7,0.8) .. controls (2.3,0.8) .. (1,1) ;
   \draw (3.2,0.9) node[anchor = south west] {{$D_{\edge'}$}} ;
   \draw (2.6,1.4) node[rotate = 10, anchor = south west] {$\edge'=K|M$} ;
   \draw [color=black, line width=1.5pt] (1,1) -- (6,2) -- (5.4,5) -- (0.4,4) -- (1,1) ;
   \draw [color=black, line width=0.5pt] (1,1) -- (5.4,5) ;
   \draw [color=black, line width=0.5pt] (6,2) -- (0.4,4) ;
   \draw (1.1,1.6) node[anchor = south west] {{$K$}} ;
   \draw [color=black, line width=1.5pt] (1,1) -- (6,2) -- (4,-1) -- (2,-1) -- (1,1) ;
   \draw[color=black, line width=0.5pt] (4,-1) .. controls (4.2,0.2) .. (3.7,0.8) ;
   \draw[color=black, line width=0.5pt] (2,-1) .. controls (2.5,0.) .. (3.7,0.8) ;
   \draw[color=black, line width=0.5pt] (1,1) .. controls (2.3,0.8) .. (3.7,0.8) ;
   \draw[color=black, line width=0.5pt] (3.7,0.8) .. controls (4.5, 1.1) .. (6,2) ;
   \draw (2.1,-0.9) node[anchor = south west] {{$M$}} ;
   \draw [color=black, line width=1.5pt] (6,2) -- (10,2.8) -- (9.4,5.8) -- (5.4,5) ;
   \draw [color=black, line width=0.5pt] (6,2) -- (9.4,5.8) ;
   \draw [color=black, line width=0.5pt] (10,2.8) -- (5.4,5) ;
   \draw (7.3,3) node[anchor = south west] {{$L$}} ;
   \draw (5.65,4.15) node[rotate = -79, anchor = south west] {$\edge=K|L$} ;
   \draw (3.6,2.9) node[rotate = -19, anchor = south west] {$\edged=D_\edge|D_{\edge'}$} ;
\end{tikzpicture}
\caption{Primal and dual meshes for the Rannacher-Turek elements.}
\label{fig:mesh}
\end{center}
\end{figure}

Finally, a constant time step denoted by $\deltat$ is used for the time discretization, and the discrete solution will thus be computed at each time $t^n$ with $t^n := n\, \deltat$, for n varying between 0 and $N := \left\lfloor \frac{T}{\delta t} \right\rfloor$. 
An index $n$ is then used to refer to the time step, that is to say that the unknows will be denoted by $(p^n_K)_{K\in\mesh,\ 0\leq n \leq N}$, $(\rho^n_K)_{K\in\mesh,\ 0\leq n \leq N}$ and $(\bfu_\edge^n)_{\edge \in \edge,\ 0\leq n \leq N }$.

\section{The numerical scheme} \label{sec:scheme}

In the following, we will propose a scheme based on the internal energy balance formulation of the Euler equations, which reads:
\begin{align}\label{eu:eq:pb_Eint}
    \partial_t (\rho e) + \dive(\rho e \bfu)+ p\, \dive \bfu =0,
\end{align}
rather than the total energy balance (Equation \eqref{eu:eq:pb_Etot}).
The former may be derived from the latter thanks to a kinetic energy identity obtained by taking the inner product of Equation \eqref{eu:eq:pb_mom} with the velocity.
Such a choice is motivated by the following reasons.
First, a well-thought approximation of the convection term of Equation \eqref{eu:eq:pb_Eint} leads to a conservation of the positivity of the internal energy, which seems difficult to obtain by a direct discretization of the the total energy balance.
Second, on a staggered discretization, the total energy is a function of quantities given on both the primal and the dual mesh.
A discretization of this equation may thus yield an awkward combination of those quantities, and this difficulty might be avoided by working directly with the internal energy balance.

The scheme thus takes the following form; it features several discrete operators, defined either on the primal mesh or on the dual meshes associated to the components of the velocity; these are defined in the following paragraphs.
\begin{subequations}\label{scheme}
\begin{align}
    & \mbox{\emph{Initialization:}} \nonumber \\
    &\hspace{2ex} 
    \rho^0_K =  \frac 1 {|K|}  \int_K  \rho_{0}(\bfx) \dx, \qquad\bfu^0_{\edge} =  \frac 1 {|D_\edge|}   \int_\Ds  \bfu_{0}(\bfx) \dx   \label{scheme:init}
    \\[2ex]
    \nonumber & \mbox{\emph{Solve for }}  n \geq 0,
    \\
    \displaystyle \label{eu:eq:scheme_mass} 
    & \forall K \in \mesh, 
    \qquad\dfrac{1}{\deltat}(\rho^{n+1}_K-\rho^n_K) + \dive_K(\rho^n \bfu^n)=0,
    \displaybreak[1] 
    \\ 
    \label{eu:eq:scheme_Eint} 
    &
    \begin{multlined}[b][10.5cm] 
        \forall K \in \mesh, \qquad \dfrac{1}{\deltat}(\rho^{n+1}_K e^{n+1}_K-\rho^n_K e^n_K) + \dive_K(\rho^n e^n \bfu^n )
        + p^n_K \, \dive_K \bfu^n = S^n_K,
    \end{multlined}
    \displaybreak[1] 
    \\[2ex] 
    \label{eu:eq:scheme_eos} 
    &\forall K \in \mesh, 
    \qquad   p^{n+1}_K=(\gamma-1)\ \rho^{n+1}_K\ e^{n+1}_K,
    \displaybreak[1] 
    \\[2ex] 
    \displaystyle \label{eu:eq:scheme_mom}  
    & 
    \begin{multlined}[b][10.5cm] 
        \forall 1 \leq i \leq d,\ \forall \edge \in \edges, \qquad
        \dfrac{1}{\deltat}(\rho^{n+1}_\Ds u_{i,\edge}^{n+1} -\rho^n_\Ds u_{i,\edge}^{n}) + \dive_\edge(\rho^n u_i^n \bfu^n) + (\eth_i p)^{n+1}_\edge =0.
    \end{multlined} 
\end{align}
\end{subequations}

\subsection{Discrete mass balance (Eqn. \eqref{eu:eq:scheme_mass})}
Let us first address the discretization of the mass balance equation \eqref{eq:cont_mass_balance}. 
The discrete mass balance \eqref{eu:eq:scheme_mass} is of finite volume type, and is set out on the primal cells, since the density unknowns are associated with these cells.
The divergence term is thus obtained by defining primal numerical fluxes $F^n_{K,\edge}$ for each time step $t^n$ across each faces $\edge$ outward of a cell $K \in \mesh$ as 
\[
\forall \edge = K|L \in \edgesint, \quad F^n_{K,\edge}=|\edge|\ \rho_\edge^n \bfu_\edge^n \cdot \bfn_{K,\edge},
\]
with $\bfn_{K,\edge}$ the normal vector to the face $\edge$ outward $K$ and $\rho_\edge^n$ a discretization of the density at the face.
Several choices are possible for this approximation at the face, as the upwind one, that is given by:
\begin{align*}
    \rho^n_\edge = \left| \begin{aligned} &
    \rho^n_K \qquad \hbox{if } \bfu^n_\edge \cdot \bfn_{K,\edge} \geq 0,
    \\[1ex] &
    \rho^n_L \qquad \hbox{otherwise.}
    \end{aligned}
\right.
\end{align*}
More precise techniques can also be derived, as for instance the MUSCL method of \cite{gas-18-mus}.
The divergence term of the discrete mass balance equation \eqref{eu:eq:pb_mom} then reads:
\begin{equation}\label{eq:disc_mass_balance}
    \dive_K(\rho^n \bfu^n) =\frac{1}{|K|}\sum_{\edge \in \edges(K)} F^n_{K,\edge}.
\end{equation}

\subsection{Discrete internal energy balance (Eqn. \eqref{eu:eq:scheme_Eint})}
We now address the discrete internal energy equation \eqref{eu:eq:scheme_Eint}.
The convection term reads
\begin{align*}
    \dive_K(\rho^n \bfu^n e^n) = \frac{1}{|K|}\sum_{\edge \in \edges(K)} F_{K,\edge}^n e^n_\edge,
\end{align*}
where the value at a face $\edge=K|L \in \edgesint$ of the internal energy may be once again obtained by the MUSCL approximation of \cite{gas-18-mus}, or by a simple upwind approximation, that is
\begin{align*}
    e^n_\edge = \left| \begin{aligned} &
    e^n_K \qquad \hbox{if } F_{K,\edge}^n \geq 0,
    \\[1ex] &
    e^n_L \qquad \hbox{otherwise.}
    \end{aligned}
\right.
\end{align*}
The divergence term is derived in the same fashion as the divergence term of the mass balance equation, and is thus given by:
\begin{align*}
    \dive_K \bfu^n = \frac{1}{|K|}\sum_{\edge \in \edges(K)} \bfu_\edge^n \cdot \bfn_{K,\edge}.
\end{align*}
Finally, the term $S^n_K$ is a corrective term which is added to ensure a kinetic energy stability result; its definition depends on the convection operator of the discrete momentum equation, or rather the technique used for the definition of the velocity on the dual faces.
We refer to \cite{gas-18-mus,her-18-cons} for its exact definition and for more details on its construction.

The values of the unknowns at the time step $n=0$ are given by an average of the initial data:
\begin{align}
    \text{For } 1\leq i \leq d, \ \forall \edge \in \edges, & \qquad u^0_{\edge,i}=\frac{1}{|D_\edge|}\int_{D_\edge} (\bfu_0(x))_idx,
\end{align}
where $\bfu_0$ is the initial data.

\subsection{Discrete momentum balance (Eqn. \eqref{eu:eq:scheme_mom})}
We now turn to the discrete momentum balance equation \eqref{eu:eq:pb_mom}.
The discrete derivative $(\eth_i p)^{n}_{\edge}$ is built at each face $\edge \in \edges$ and for all component $1 \leq i \leq d$, and is given as follows:
\begin{align}
    \forall \edge \in \edgesint, \edge = K|L, \qquad (\eth_i p)^{n}_{\edge}=\frac{|\edge|}{|D_\edge|}(p_L-p_K)\bfn_{K,\edge}\cdot\bfe^{(i)},
\end{align}
where $\bfe^{(i)}$ is the $i^{th}$ vector of the orthonormal basis of $\xR^d$, and $\bfn_{K,\edge}$ the normal vector to the face $\edge$ outward the cell $K$.

The discrete divergence operator $\dive_\edge(\rho^n u_i^n \bfu^n)$ is of finite-volume type, and, since the discretization is staggered, relies on the dual mesh.
It takes the general form:
\begin{equation}\label{eq:conv_vel}
    \mbox{for } 1 \leq i \leq d,\forall \edge \in \edges, \qquad
    \dive_\edge(\rho^n u_i^n \bfu^n) =  \frac{1}{|D_\edge|}\sum_{\edged \in \edgesd(D_\edge)} F^n_{\edge,\edged}\ u_{i,\edged}^n
\end{equation}
where $u_{\edged,i}^n$ is an approximation of $u_i$ over the face $\edged$ at time $t_n$, and $F^n_{\edge,\edged}$ is a mass flux leaving $D_\edge$ through the dual face $\edged$ at time $t^n$.

The expression of the velocity at the dual face, \ie of the quantity $u_{i,\edged}^n$ for $\edged = \edge|\edge'$ in Equation \eqref{eq:conv_vel} may be obtained thanks to an upwind method, that is :
\begin{align*}
    \mbox{for } 1 \leq i \leq d, \qquad u_{i,\edged}^n = \left|
        \begin{array}{ll}
            u_{i,\edge}^n & \text{if } F^n_{\edge,\edged} \geq 0 \\
            u_{i,\edge'}^n & \text{otherwise,} 
        \end{array}
    \right.
\end{align*} 
In order to enhance the accuracy of the scheme, one may also choose
an algebraic MUSCL technique, we refer to \cite{bru-22-mus} for details on such a procedure.

\medskip
The derivation of the quantities $\rho_\Ds^{n+1}$, $\rho_\Ds^n$ and $F^n_{\edge,\edged}$ is performed to ensure a finite volume mass balance over the dual cells:
\begin{equation}\label{eq:mass_D}
    \mbox{for } 1 \leq i \leq d,\ \forall \edge \in \edges, \qquad
    \frac{|D_\edge|}{\deltat} \ (\rho_{D_\edge}^{n+1}-\rho^n_{D_\edge})
    + \sum_{\edged\in \edgesd(D_\edge)} F^n_{\edge,\edged}=0.
\end{equation}
For $\edge$ in $\edgesint$ such that $\edge=K|L$, the approximate densities on the dual cell $D_\edge$ are given (at any time level) by the following weighted average:
\begin{equation} \label{eq:pd2}
    |D_\edge|\ \rho_\Ds= \xi_K^\edge|K|\ \rho_K + \xi_L^\edge|L|\ \rho_L, \qquad \mbox{with } \xi_K^\edge = \frac{|D_{K,\edge}|}{|K|}, \, K\in\mesh, ~\edge\in\edges(K).
\end{equation}
For a (half) diamond cell associated to an external face, the density is equal to the density in the adjacent primal cell.
Indeed, this identity is necessary to ensure a discrete kinetic energy balance, and the existence of a such balance equation is central to obtain consistent schemes (see \cite{her-14-som,her-18-cons,her-20-cons}).
The construction of the dual fluxes $F^n_{\edge,\edged}$ for different grid cells is the subject of the next section.

\section{Construction of the dual fluxes from the primal fluxes} \label{sec:vect_conv_op}

In the following, we drop the time exponent for clarity.
We first recall the conditions that must satisfy the set of dual fluxes $(F_{\edge,\edged})_{\edged\subset K}$, is computed by solving a linear system depending on the primal fluxes $(F_{K,\edge})_{\edge\in\edges(K)}$ appearing in the discrete mass balance, in order to obtain the stability of the resulting non linear convection operator.

\begin{definition}[Constraints of the dual fluxes \cite{ans-11-anl}]\label{def:dualfluxes}
The fluxes through the faces of the dual mesh are defined so as to satisfy the following three constraints:
\begin{itemize}
\item[(H1)] The discrete mass balance over the half-diamond cells is satisfied, in the following sense. 
For any primal cell $K$ in $\mesh$, the set $(\fluxd)_{\edged\subset K}$ of dual fluxes included in $K$ solves the following linear system
\begin{equation}\label{eq:F_syst}
    F_{K,\edge} + \sum_{\edged \in \edgesd(D_\sigma),\ \edged \subset K} F_{\edge,\edged}= \xi_K^\edge \ \sum_{\edge' \in \edges(K)} F_{K,\edge'}, \qquad \edge \in \edges(K).
\end{equation}
\item[(H2)] The dual fluxes are conservative, \ie~ for any dual face $\edged=D_\edge|D_\edge'$, we have $F_{\edge,\edged}=-F_{\edge',\edged}$.
\item[(H3)] The dual fluxes are bounded with respect to the primal fluxes $(F_{K,\edge})_{\edge \in \edges(K)}$ in the sense that there exists a universal constant real number $C$ such that:
\begin{align}\label{eq:F_bounded}
    |F_{\edge,\edged}| \leq C \ \max \,\left \lbrace |F_{K,\edge}|, \qquad \edge \in \edges(K) \right \rbrace,  K \in \mesh,\  \edge \in \edges(K),\  \edged \in \edgesd(D_\sigma),\ \edged \subset K .
\end{align}
\end{itemize}
\end{definition}

The assumptions (H1)-(H3) are sufficient to imply the consistency of the discrete convection operator, as we shall see later on.
Note however that the system of equations \eqref{eq:F_syst} is singular and has an infinite number of solutions;  the additional constraint \eqref{eq:F_bounded} is needed for stability purposes, but the system \eqref{eq:F_syst}-\eqref{eq:F_bounded} remains underdetermined.
The choice of a particular solution to this system depends on the type of grid cell; it is detailed in the next paragraphs.
Note that, since \eqref{eq:F_syst} is a linear system for the dual mass fluxes, a solution of \eqref{eq:F_syst} may be expressed as:
\begin{equation}
    F_{\edge,\edged} = \sum_{\edge' \in \edges(K)} \alpha_K^\edged F_{K,\edge'},\qquad \edge \in \edges(K),\
    \edged \in \edgesd(D_{\edge}) \mbox{ and } \edged \subset K,
\label{eqdef:dualflux}
\end{equation}
and the constraint \eqref{eq:F_bounded} amounts to requiring the coefficients $(\alpha_K^\edged)_{\edge,\edge'\in \edges(K)}$ to be bounded by a universal constant.
In practice, the coefficients used in our numerical experimentations and derived below satisfy $|\alpha_K^\edged|\leq 1$ for all $\edge,\edge'\in \edges(K)$ and for all $K\in\mesh$.

We are thus able to cope with a quite general definition of the diamond cells, since, up to now,  their volume itself is not specified.
As mentioned in Section \ref{sec:mesh}, in order to simplify the implementation, we however choose in practice to impose that the half-diamond cells of a primal cell $K$ all have the same measure:
\begin{equation}
  |D_{K,\edge}|=|K|/\mathrm{card}(\edges(K)). \label{equal-volume}
\end{equation}
 
Therefore, the real number $\xi_K^\edge$ in \eqref{eq:pd2} and \eqref{eq:F_syst} takes the following values: 
\[
 \xi_K^\edge = \begin{cases}
                 1/3 \mbox{ for 2D simplices,}\\
                 1/4 \mbox{ for 2D quadrangles and 3D simplices,}\\
                 1/5 \mbox{ for quadrangle-based pyramids and triangular prisms,}\\
                 1/6 \mbox{ for hexahedra,}
               \end{cases}
\]
As a consequence, the system \eqref{eq:F_syst} only depends on the shape of the cell $K$ under consideration.
We may thus consider a particular geometry for $K$, and find an expression for the coefficients $((\alpha_K^\edged)_{\edge,\edge'\in \edges(K)}$ which we  apply to all similar cells, thus automatically satisfying the constraint \eqref{eq:F_bounded}.

\subsection{The case of a geometric dual mesh}\label{sec:geom}

We first consider the case when the dual cells satisfying \eqref{equal-volume} can be geometrically built on a cell $K$ in the sense that an explicit definition of $D_{K,\edge}$ can be given for all $\edge \in \edges(K)$.
This is for instance the case when $K$ is a simplex, a square or a cube.
Let us consider a momentum field $\bfw$ with constant divergence, such that:
\[
    \int_\edge \bfw \cdot \bfn_{K,\edge} \dedge(\bfx) = F_{K,\edge}, \qquad \forall \edge \in \edges(K).
\]
Then an easy computation shows that the definition
\begin{align} \label{eq:dual_flux_geom}
    F_{\edge,\edged} = \int_\edged \bfw \cdot \bfn_{\edge,\edged} \dedge(\bfx),
\end{align}
where the unit vector normal to $\edged $ outward $D_\edge$ is denoted by $\bfn_{\edge,\edged}$, satisfies \eqref{eq:F_syst} (see \cite[Lemma 3.2]{ans-11-anl}).

For instance, suppose that $K$ is the reference square $K = [0,1]^2$.
Next, denote by $W,E,S,N$ the primal faces of $K$ and $F_W, F_E, F_S$ and $F_N$ their associated primal fluxes, as in Figure \ref{tab:quadrilateral_dual_fluxes}.
Then, it might be possible to choose the momentum field $\bfw$ as:
\begin{align*}
    \bfw(x,y) = 
    \begin{bmatrix}
        (1-x)(-F_W) + x F_E \\
        (1-y)(-F_S) + y F_N
    \end{bmatrix},
\end{align*}

Next, we need to examine the connectivity inside the cell $K$.
More precisely, for a given interface $\edge$ of $K$, $\edge \in \{W,E,S,N\}$, we first list all the interfaces $\edge' \in \{W,E,S,N\} \backslash \{ \edge \}$ such that $D_{K,\edge}$ and $D_{K,\edge'}$ are neighbouring dual cells.
Let us denote by $F_{\edge,\edged}$ denote the dual flux coming from the diamond cell $D_{K,\edge}$ towards the diamond cell $D_{K,\edge'}$ with $\edged = \edge|\edge'$.
For each $\edge \in \edges, \edged \in \edgesd(D_{\edge}), \edged \subset K$, we then obtain thanks to \eqref{eq:dual_flux_geom}:
\begin{align*}
    F_{\edge,\edged} = \alpha_W^{\edged} F_W + \alpha_E^\edged F_E + \alpha_S^\edged F_S + \alpha_N^\edged F_N.
\end{align*}
The set of coefficients $\{\alpha_K^\edged, \edged= \edge|\edge'$ with $ \edge,\edge' \in\edges(K)\}$ obtained in each case is given on Figure \ref{tab:quadrilateral_dual_fluxes}.
The coefficients for simplicial cells are given on Fig. \ref{tab:triangular_dual_fluxes} and  Fig. \ref{tab:tetrahedral_dual_fluxes}, and the coefficient for hexahedral cells are given on Tab. \ref{tab:hexahedral_dual_fluxes}.

\begin{figure}[ht!]
\begin{minipage}{0.45\textwidth}
\centering
    \begin{tikzpicture}
        \newcommand{\Left}{1,0}
        \newcommand{\Right}{4,0}
        \newcommand{\Top}{1,3}
        
        \draw (\Top) -- (\Left) ;
        \draw (\Top) -- (\Right) ;
        \draw (\Left) -- (\Right) ;

        \draw[->, line width=2pt, color=gray] (2.5,0) -- (2.5,-1) ;
        \draw (2.3,-0.7) node[rotate = 90, anchor = south west, inner sep =.5 ] {\textcolor{gray}{$F_{\rm S}$}} ;
        \draw[->, line width=2pt, color=gray] (1,1.5) -- (0,1.5) ;
        \draw (0.3,1.7) node[anchor = south west, inner sep =.5] {\textcolor{gray}{$F_{\rm E}$}} ;
        \draw[->, line width=2pt, color=gray] (2.5,1.5) -- (3.25,2.25) ;
        \draw (2.4,1.75) node[rotate = 45, anchor = south west, inner sep =.5] {\textcolor{gray}{$F_{\rm W}$}} ;
    \end{tikzpicture}
\end{minipage}
\begin{minipage}{0.45\textwidth}
    \begin{center}
    \begin{tabular}{|c|c|c|c|}
        \hline
        $ \edged= \edge|\edge' $& $\alpha_W^{\edged}$ & $\alpha_E^{\edged}$ & $\alpha_S^{\edged}$ \\ \hline
        $W|S$ & $\rule[-1.5mm]{0mm}{5mm} -1/3$ & $0  $  & $1/3$  \\ \hline
        $S|E$ & $\rule[-1.5mm]{0mm}{5mm} 0   $ & $1/3$  & $-1/3$ \\ \hline
        $E|W$ & $\rule[-1.5mm]{0mm}{5mm}  1/3$ & $-1/3$ & $0   $ \\ \hline
    \end{tabular}
    \begin{align*}
        F_{\edge,\edged} = & \alpha_W^{\edged} F_W + \alpha_E^{\edged} F_E + \alpha_S^{\edged} F_S + \\ & + \alpha_N^{\edged} F_N.
    \end{align*}
    \end{center}
\end{minipage}
\caption{Primal flux over a triangular cell and coefficient of the dual fluxes}
\label{tab:triangular_dual_fluxes}
\end{figure}

\begin{figure}[ht!]
\begin{minipage}{0.4\textwidth}
\centering
    \begin{tikzpicture}
        \newcommand{\TriLeftBack}{2.2,1}
        \newcommand{\TriLeftFront}{1,0}
        \newcommand{\TriLeftTop}{1,3}
        \newcommand{\TriRightFront}{4,0}
        
        \draw (\TriLeftTop) -- (\TriLeftFront) ;
        \draw [dashed](\TriLeftFront) -- (\TriLeftBack) -- (\TriLeftTop) ;
        \draw (\TriLeftFront) -- (\TriRightFront) ;
        \draw (\TriLeftTop) -- (\TriRightFront) ;
        \draw [dashed](\TriLeftBack) -- (\TriRightFront) ;
        
        \draw [line width=2pt, color=gray, dashed](2.3,0.5) -- (2.3,-0) ;
        \draw [->, line width=2pt, color=gray](2.3,00) -- (2.3,-0.5) ;
        \draw (2.45,-1) node[rotate = 90, anchor = south west, inner sep =.5] {\textcolor{gray}{$F_{\rm B}$}} ; 
        \draw [line width=2pt, color=gray, dashed](1.46,1.46) -- (1,1.46) ;
        \draw [->, line width=2pt, color=gray](1,1.46) -- (0.46,1.46) ;
        \draw (-0.3,1.26) node[ anchor = south west, inner sep =.5] {\textcolor{gray}{$F_{\rm W}$}} ;
        \draw [->, line width=2pt, color=gray] (2,1) -- (1.28,0.4) ;
        \draw (0.5,0.15) node[anchor = south west, inner sep =.5]{\textcolor{gray}{$F_{\rm S}$}} ;
        \draw [line width=2pt, color=gray, dashed] (2.4,1.16) -- (2.8,1.5) ;
        \draw [->, line width=2pt, color=gray] (2.8,1.5) -- (3.1,1.75) ;
        \draw (3.2,1.83) node[anchor = south west, inner sep =.5]{\textcolor{gray}{$F_{\rm E}$}} ;
    \end{tikzpicture}
\end{minipage}
\begin{minipage}{0.55\textwidth}
    \begin{center}
    \begin{tabular}{|c|c|c|c|c|}
        \hline
        $ \edged= \edge|\edge'$ & $\alpha_W^\edged$ & $\alpha_E^\edged$ & $\alpha_S^\edged$ & $\alpha_B^\edged$ \\ \hline
        $ B|S$ & $\rule[-1.5mm]{0mm}{5mm} 0   $ & $0   $  & $ 1/4$ & $-1/4$ \\ \hline
        $ B|W$ & $\rule[-1.5mm]{0mm}{5mm}  1/4$ & $0   $  & $0   $ & $-1/4$ \\ \hline
        $ B|E$ & $\rule[-1.5mm]{0mm}{5mm} 0   $ & $ 1/4$  & $0   $ & $-1/4$ \\ \hline
        $ S|W$ & $\rule[-1.5mm]{0mm}{5mm}  1/4$ & $0   $  & $-1/4$ & $0   $ \\ \hline
        $ W|E$ & $\rule[-1.5mm]{0mm}{5mm} -1/4$ & $ 1/4$  & $0   $ & $0   $ \\ \hline
        $ E|S$ & $\rule[-1.5mm]{0mm}{5mm} 0   $ & $-1/4$  & $ 1/4$ & $0   $ \\ \hline
    \end{tabular}
    \begin{align*}
        F_{\edge,\edged} = & \alpha_W^{\edged} F_W + \alpha_E^{\edged} F_E + \alpha_S^{\edged} F_S + \alpha_N^{\edged} F_N \\
        & + \alpha_B^{\edged} F_B
    \end{align*}
    \end{center}
\end{minipage}%
\caption{Primal flux over a tetrahedral cell and coefficient of the dual fluxes}
\label{tab:tetrahedral_dual_fluxes}\end{figure}
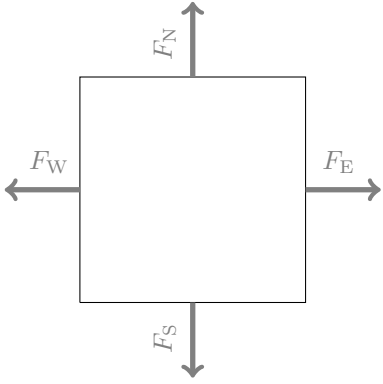

\begin{figure}
\begin{minipage}{0.43\textwidth}
\centering
    \begin{tikzpicture}
        \newcommand{\LeftBottom}{1,0}
        \newcommand{\LeftTop}{1,3}
        \newcommand{\RightBottom}{4,0}
        \newcommand{\RightTop}{4,3}
        
        \draw (\LeftBottom) -- ( \LeftTop) -- ( \RightTop) -- ( \RightBottom) -- ( \LeftBottom) ;
            
        \draw[line width=2pt, color=gray, -> ](2.5,0) -- ( 2.5,-1) ;
        \draw(2.3,-0.7) node[ rotate = 90, anchor = south west, inner sep =.5] {\textcolor{gray}{$F_{\rm S}$}} ;
        \draw[line width=2pt, color=gray, -> ](4,1.5) -- ( 5,1.5) ;
        \draw(4.2,1.7) node[ anchor = south west, inner sep =.5] {\textcolor{gray}{$F_{\rm E}$}} ;
        \draw[line width=2pt, color=gray, -> ](2.5,3) -- ( 2.5,4) ;
        \draw(2.3,3.2) node[ rotate = 90, anchor = south west, inner sep =.5]{\textcolor{gray}{$F_{\rm N}$}} ;
        \draw[line width=2pt, color=gray, -> ](1,1.5) -- ( 0,1.5) ;
        \draw(0.3,1.7) node[ anchor = south west, inner sep =.5] {\textcolor{gray}{$F_{\rm W}$}} ;
    \end{tikzpicture}
\end{minipage}
\begin{minipage}{0.55\textwidth}
    \begin{center}
    \begin{tabular}{|c|c|c|c|c|}
        \hline
        $\edged= \edge|\edge'$ & $\alpha_W^\edged$ & $\alpha_E^\edged$ & $\alpha_S^\edged$ & $\alpha_N^\edged$
        \\ \hline
        $W|S$ & $\rule[-1.5mm]{0mm}{5mm}-3/8$ & $1/8$  & $3/8$  & $-1/8$ 
        \\ \hline
        $S|E$ & $\rule[-1.5mm]{0mm}{5mm}-1/8$ & $3/8$  & $-3/8$ & $1/8$
        \\ \hline
        $E|N$ & $\rule[-1.5mm]{0mm}{5mm} 1/8$ & $-3/8$ & $-1/8$ & $3/8$
        \\ \hline
        $N|W$ & $\rule[-1.5mm]{0mm}{5mm} 3/8$ & $-1/8$ & $1/8$  & $-3/8$
    \\ \hline
    \end{tabular}
    \begin{align*}
        F_{\edge,\edged} = \alpha_W^{\edged} F_W + \alpha_E^\edged F_E + \alpha_S^\edged F_S + \alpha_N^\edged F_N.
    \end{align*}
    \end{center}
\end{minipage}%
\caption{Primal flux over a quadrilateral cell and coefficient of the dual fluxes}
\label{tab:quadrilateral_dual_fluxes}
\end{figure}

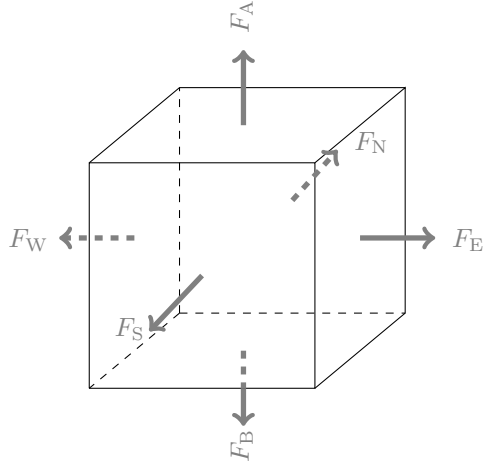
\begin{figure}[ht!]\label{fig:hexahedral_dual_fluxes}
    \centering
	\begin{tikzpicture}
        \newcommand{\LeftFrontBottom}{0,0}
        \newcommand{\LeftBehindBottom}{1.2,1}
        \newcommand{\LeftFrontTop}{0,3}
        \newcommand{\LeftBehindTop}{1.2,4}
        \newcommand{\RightFrontBottom}{3,0}
        \newcommand{\RightBehindBottom}{4.2,1}
        \newcommand{\RightFrontTop}{3,3}
        \newcommand{\RightBehindTop}{4.2,4}
        
        \draw(\LeftFrontBottom) -- ( \LeftFrontTop) -- ( \RightFrontTop) -- ( \RightFrontBottom) -- ( \LeftFrontBottom) ;
        \draw(\RightFrontTop) -- ( \RightBehindTop) -- ( \RightBehindBottom) -- ( \RightFrontBottom) ;
        \draw(\RightBehindTop) -- ( \LeftBehindTop) -- ( \LeftFrontTop) ;
        \draw[dashed](\LeftBehindBottom) -- ( \LeftBehindTop) ;
        \draw[dashed](\LeftBehindBottom) -- ( \LeftFrontBottom) ;
        \draw[dashed](\LeftBehindBottom) -- ( \RightBehindBottom) ;
            
        \draw[line width=2pt, color=gray, dashed](2.05,0.5) -- ( 2.05,-0) ;
        \draw[line width=2pt, color=gray, -> ](2.05,00) -- ( 2.05,-0.5) ;
        \draw(2.2,-1) node[ rotate = 90, anchor = south west, inner sep =.5] {\textcolor{gray}{$F_{\rm B}$}} ;
        \draw[line width=2pt, color=gray, -> ](3.6,2) -- ( 4.6,2) ;
        \draw(4.8,1.8) node[ anchor = south west, inner sep =.5] {\textcolor{gray}{$F_{\rm E}$}} ;
        \draw[line width=2pt, color=gray, -> ](2.05,3.5) -- ( 2.05,4.5) ;
        \draw(2.2,4.7) node[ rotate = 90, anchor = south west, inner sep =.5] {\textcolor{gray}{$F_{\rm A}$}} ;
        \draw[line width=2pt, color=gray, dashed, -> ](0.6,2) -- ( -0.4,2) ;
        \draw(-1.1,1.8) node[ anchor = south west, inner sep =.5] {\textcolor{gray}{$F_{\rm W}$}} ;
        \draw[line width=2pt, color=gray, -> ](1.5,1.5) -- ( 0.8,0.76) ;
        \draw(0.3,0.6) node[ anchor = south west, inner sep = 1] {\textcolor{gray}{$F_{\rm S}$}} ;
        \draw[line width=2pt, color=gray, dashed, -> ](2.7,2.5) -- ( 3.3, 3.17 ) ;
        \draw(3.4, 3) node[ anchor = south west] {\textcolor{gray}{$F_{\rm N}$}} ;
    \end{tikzpicture}
\caption{Primal flux over a hexahedral cell}
\end{figure}

\begin{table}[ht!]\label{tab:hexahedral_dual_fluxes}
    \centering
    \begin{tabular}{|c|c|c|c|c|c|c|}
        \hline
        $\edged=\edge|\edge'$ & $\alpha_W^\edged$ & $\alpha_E^\edged$ & $\alpha_S^\edged$ & $\alpha_N^\edged$ & $\alpha_A^\edged$ & $\alpha_B^\edged$
        \\ \hline
        $B|S$	& $\rule[-1.5mm]{0mm}{5mm} 0$			&	$0$ 	&	$5/24$	& $-1/24$& $1/24$	& $-5/24$
        \\ \hline
        $S|A$ & $\rule[-1.5mm]{0mm}{5mm} 0$				& $0$ 	& $-5/24$	& $1/24$ & $5/24$	& $-1/24$
        \\ \hline
        $A|N$	& $\rule[-1.5mm]{0mm}{5mm} 0$			& $0$	& $-1/24$	& $5/24$ & $-5/24$	& $ 1/24$
        \\ \hline
        $N|B$	& $\rule[-1.5mm]{0mm}{5mm} 0$			& $0$	& $1/24$ 	& $-5/24$& $-1/24$	& $5/24$ 
        \\ \hline
        $W|S$	&  $\rule[-1.5mm]{0mm}{5mm} -5/24$	& $1/24$	& $5/24$ & $-1/24$& $0$		& $0$
        \\ \hline
        $S|E$	& 	$\rule[-1.5mm]{0mm}{5mm} -1/24$	& $5/24$	& $-5/24$ & $1/24$& $0$		& $0$
        \\ \hline
        $E|N$	&  $\rule[-1.5mm]{0mm}{5mm} 1/24$	& $-5/24$	& $-1/24$ & $5/24$& $0$		& $0$
        \\ \hline
        $N|W$& 	$\rule[-1.5mm]{0mm}{5mm} 5/24$	& $-1/24$	& $1/24$ & $-5/24$& $0$		& $0$
        \\ \hline
        $B|E$	&  $\rule[-1.5mm]{0mm}{5mm} -1/24$	& $5/24$	& $0$ & $0$& $1/24$		& $-5/24$
        \\ \hline
        $E|A$	& 	$\rule[-1.5mm]{0mm}{5mm} 1/24$	& $-5/24$	& $0$ & $0$& $5/24$		& $-1/24$
        \\ \hline
        $A|W$	&  $\rule[-1.5mm]{0mm}{5mm} 5/24$	& $-1/24$	& $0$ & $0$& $-5/24$		& $1/24$
        \\ \hline
        $W|B$& 	$\rule[-1.5mm]{0mm}{5mm} -5/24$	& $1/24$	& $0$ & $0$& $-1/24$		& $5/24$
        \\ \hline
    \end{tabular}
    \begin{align*}
        F_{\edge,\edged} = \alpha_W^{\edged} F_W + \alpha_E^\edged F_E + \alpha_S^\edged F_S + \alpha_N^\edged F_N + \alpha_A^\edged F_A + \alpha_B^\edged F_B.
    \end{align*}
    \caption{Coefficient of the dual fluxes for a hexahedral cell}
\end{table}

\subsection{The case of a virtual dual mesh}\label{subsec:prism}
 
For highly distorted (\ie\ far for parallelogram) quadrilaterals or hexahedra $K$ of face $\edge$, when the measure of $\edge$ is small with respect to the characteristic dimensions of $K$, it may be impossible to define $D_{K,\edge}$ as a triangle or a pyramid, respectively, since such a volume of basis $\edge$ and included in $K$ cannot satisfy $|D_{K,\edge}|=|K|/\mathrm{card}(\edges(K))$.
In such a case, the similarity of the above-defined convection operator with a standard finite volume operator is only formal: indeed, its discrete formulation does not rely on the integration of the continuous operator over a control volume, whose shape and interfaces (and normal vectors) are not defined.
For the same reasons, the method described in the above paragraph does not apply to prismatic and pyramidal cells.
We thus present an algebraic technique to deal with those kinds of cells.

\medskip

Let us suppose that $K$ is a prismatic cell : we give a constructive process to deduce the dual fluxes directly from the system \eqref{eq:F_syst}-\eqref{eq:F_bounded}.
Recall that a solution is not unique since the system is underdetermined, but it is still possible to deduce at least one solution.
To fix ideas, we consider a prism with a triangular basis and denote by $S$, $E$, $N$, $W$ and $B$ the faces of the prism, so that $F_S,F_E,F_N,F_W$ and $F_B$ are the primal fluxes over the faces of the prism, as can be seen on Figure \ref{tab:prism_primal_flux}. 
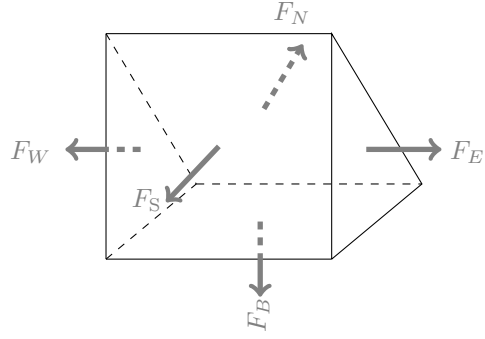
\begin{figure}[ht!]\label{fig:prism_primal_flux}
    \centering
    \begin{tikzpicture}
        \newcommand{\TriLeftBack}{1.2,1}
        \newcommand{\TriLeftFront}{0,0}
        \newcommand{\TriLeftTop}{0,3}
        \newcommand{\TriRightBack}{4.2,1}
        \newcommand{\TriRightFront}{3,0}
        \newcommand{\TriRightTop}{3,3}
		
        \draw (\TriLeftTop) -- (\TriLeftFront) ;
        \draw [dashed](\TriLeftFront) -- (\TriLeftBack) -- (\TriLeftTop) ;
        \draw (\TriRightFront) -- (\TriRightBack) -- (\TriRightTop) -- (\TriRightTop) -- (\TriRightFront) ;
        \draw (\TriLeftFront) -- (\TriRightFront) ;
        \draw (\TriLeftTop) -- (\TriRightTop) ;
        \draw [dashed](\TriLeftBack) -- (\TriRightBack) ;
            
        \draw [line width=2pt, color=gray, dashed](2.05,0.5) -- (2.05,-0) ;
        \draw [line width=2pt, color=gray, ->](2.05,00) -- (2.05,-0.5) ;
        \draw (2.2,-1) node[rotate = 90, anchor = south west, inner sep =.5] {\textcolor{gray}{$F_B$}} ;
        \draw [line width=2pt, color=gray, ->](3.46,1.46) -- (4.46,1.46) ;
        \draw (4.55,1.26) node[anchor = south west, inner sep =.5] {\textcolor{gray}{$F_E$}} ;
        \draw [line width=2pt, color=gray, dashed](0.46,1.46) -- (0,1.46) ;
        \draw [line width=2pt, color=gray, ->](0,1.46) -- (-0.54,1.46) ;
        \draw (-1.3,1.26) node[anchor = south west, inner sep =.5] {\textcolor{gray}{$F_W$}} ;
        \draw [line width=2pt, color=gray, ->](1.5,1.5) -- (0.8,0.76) ;
        \draw(0.3,0.6) node[ anchor = south west, inner sep = 1] {\textcolor{gray}{$F_{\rm S}$}} ;
        \draw [line width=2pt, color=gray, dashed, ->](2.1,2) -- (2.615,2.857) ;
        \draw (2.2,3.1) node[anchor = south west, inner sep =.5] {\textcolor{gray}{$F_N$}} ;
    \end{tikzpicture}
\caption{Primal fluxes over a prism}
\end{figure}

\begin{table}[ht!]\label{tab:prism_primal_flux}
    \centering
        \begin{tabular}{|l|c|c|c|c|c|}
        \hline
        $\edged=\edge|\edge'$ & $\alpha_B^\edged$ & $\alpha_S^\edged$ & $\alpha_N^\edged$ & $\alpha_E^\edged$ & $\alpha_W^\edged$ \\ \hline
        $S|N$ &    0 & -1/5 &  1/5 &     0 &     0 \\ \hline
        $N|B$ &  1/5 &    0 & -1/5 &     0 &     0 \\ \hline
        $B|S$ & -1/5 &  1/5 &    0 &     0 &     0 \\ \hline
        $E|B$ &  1/5 &    0 &    0 & -4/15 &  1/15 \\ \hline
        $E|S$ &    0 &  1/5 &    0 & -4/15 &  1/15 \\ \hline
        $E|N$ &    0 &    0 &  1/5 & -4/15 &  1/15 \\ \hline
        $W|B$ &  1/5 &    0 &    0 &  1/15 & -4/15 \\ \hline
        $W|S$ &    0 &  1/5 &    0 &  1/15 & -4/15 \\ \hline
        $W|N$ &    0 &    0 &  1/5 &  1/15 & -4/15 \\ \hline
    \end{tabular}
    \begin{align*}
       F_{\edge,\edged} = \alpha_W^{\edged} F_W + \alpha_E^\edged F_E + \alpha_S^\edged F_S + \alpha_N^\edged F_N +  \alpha_B^\edged F_B.
    \end{align*}
    \caption{Coefficient of the dual fluxes for a prism}
\end{table}

Then, we denote once again for $\edge,\edge' \in \{D,S,E,N,W\}^2, \edge \neq \edge'$ by $F_{\edge,\edge'}$ the dual flux coming from the diamond cell associated with the face $\edge$ towards the diamond cell associated with the face $\edge'$, accordingly to the connectivity of the cell.
Using the conservativity of the dual flux (that is $F_{\edge,\edged}=-F_{\edge',\edged}$), the system \ref{eq:F_syst} may be written for the prism as:
\begin{align} \label{eq:dual_flux_syst_alg}
    A F_{d} = B F_{p} 
\end{align}
with 
\begin{align*}
    & F_{p} = ( F_U, F_S, F_N, F_E, F_W )^t, \\
    & F_{d} = (F_{S|N}, F_{N|U}, F_{U|S}, F_{E|U}, F_{E|S},F_{E|N},F_{W|U},F_{W|S},F_{W|N})^t,
\end{align*}
where $F_{\sigma|\sigma'}$  denotes the dual flux coming from the diamond cell $D_{\edge}$   towards the diamond cell  $D_{\edge'}$ associated with the face $\edge'$, and
\begin{align*}
    & A=
    \begin{pmatrix}
    0& -1&  1& -1&  0&  0& -1&  0&  0 \\
    1&  0& -1&  0& -1&  0&  0& -1&  0 \\
    -1& 1&  0&  0&  0& -1&  0&  0& -1 \\
    0&  0&  0&  1&  1&  1&  0&  0&  0 \\
    0&  0&  0&  0&  0&  0&  1&  1&  1
    \end{pmatrix},
    \\
    & B= 
    \begin{pmatrix}
    -4/5 &  1/5 &  1/5 &  1/5 &  1/5 \\
     1/5 & -4/5 &  1/5 &  1/5 &  1/5 \\
     1/5 &  1/5 & -4/5 &  1/5 &  1/5 \\
     1/5 &  1/5 &  1/5 & -4/5 &  1/5 \\
     1/5 &  1/5 &  1/5 &  1/5 & -4/5
    \end{pmatrix}.
\end{align*}
Since the dual fluxes are supposed to be a linear combination of the primal fluxes, the vector $F_{d}$ may be expressed as $F_{d} = X F_{p}$, where $X$ is a matrix whose entries are the desired coefficients $\{\alpha_\edge^\edged, \edge\in \edges(K), \edged = \edge|\edged, \edged \in  \edges(K), \edged \ne\edge\}$.
Substituting this expression in \eqref{eq:dual_flux_syst_alg}, one obtains an underdetermined system $A X = B$, for which a solution are obtained through the least-square method.
The coefficients thus obtained are given in the Figure \ref{fig:prism_primal_flux}. 

The same technique is applied for a pyramidal mesh.
The notation for its faces and its primal flux is the same as before, as can be seen in Figure \ref{fig:pyramid_fluxes}.
Note that the connectivity of the pyramid is similar to the one of the prism, up to the fact that the northern and southern dual cells are not connected.
The linear system for a pyramidal mesh may thus be deduced from the prism one by removing any instance of the dual flux $F_{S|N}$.
The coefficients are given in Table \ref{fig:pyramid_fluxes}.

\begin{figure}[ht!]\label{fig:pyramid_fluxes}
    \centering
    \begin{tikzpicture}
        \newcommand{\LeftBack}{1.2,1}
        \newcommand{\LeftFront}{0,0}
        \newcommand{\Top}{0,3}
        \newcommand{\RightBack}{4.2,1}
        \newcommand{\RightFront}{3,0}
    
        \draw (\LeftFront) -- (\RightFront) -- (\RightBack) ;
        \draw [dashed](\RightBack) -- (\LeftBack) -- (\LeftFront) ;
        \draw (\Top) -- (\LeftFront) ;
        \draw (\Top) -- (\RightFront) ;
        \draw (\Top) -- (\RightBack) ;
        \draw [dashed](\Top) -- (\LeftBack) ;
        
        \draw [line width=2pt, color=gray, dashed](2.05,0.5) -- (2.05,-0) ;
        \draw [line width=2pt, color=gray, ->](2.05,00) -- (2.05,-0.5) ;
        \draw (2.2,-1) node[rotate = 90, anchor = south west, inner sep =.5] {\textcolor{gray}{$F_{\rm B}$}} ;
        \draw [line width=2pt, color=gray, dashed](0.46,1.46) -- (0,1.46) ;
        \draw [line width=2pt, color=gray, ->](0,1.46) -- (-0.54,1.46) ;
        \draw (-1.3,1.26) node[anchor = south west, inner sep =.5] {\textcolor{gray}{$F_{\rm W}$}} ;
        \draw [line width=2pt, color=gray, ->](1,1) -- (0.28,0.4) ;
        \draw (-0.5,0.15) node[anchor = south west, inner sep =.5] {\textcolor{gray}{$F_{\rm S}$}} ;
        \draw [line width=2pt, color=gray, ->](2.3,1.3) -- (3,2) ;
        \draw (3.1,2.1) node[anchor = south west, inner sep =.5] {\textcolor{gray}{$F_{\rm E}$}} ;
        \draw [line width=2pt, color=gray, dashed, ->](1.6,1.6) -- (2.151,2.52) ;
        \draw [line width=2pt, color=gray, dashed](2.151,2.52) -- (2,2.27) ;
        \draw [line width=2pt, color=gray, ->](2,2.27) -- (2.151,2.52) ;
        \draw (2.25,2.67) node[anchor = south west, inner sep =.5] {\textcolor{gray}{$F_{\rm N}$}} ;
    \end{tikzpicture}
\end{figure}

\begin{table}[ht!]
    \centering
    \begin{tabular}{|c|c|c|c|c|c|}
        \hline
        $ \edged = \edge |\edge' $ & $\alpha_B^\edged$ & $\alpha_S^\edged$ & $\alpha_E^\edged$ & $\alpha_N^\edged$ & $\alpha_W^\edged$ \\ \hline
        ${B|S}$ &  -1/5 &  4/15 &      0 &  -1/15 &     0 \\ \hline
        ${B|E}$ &  -1/5 &     0 &   4/15 &      0 & -1/15 \\ \hline
        ${B|N}$ &  -1/5 & -1/15 &      0 &   4/15 &     0 \\ \hline
        ${B|W}$ &  -1/5 &     0 &  -1/15 &      0 &  4/15 \\ \hline
        ${S|E}$ &     0 & -4/15 &   4/15 &   1/15 & -1/15 \\ \hline
        ${E|N}$ &     0 & -1/15 &  -4/15 &   4/15 &  1/15 \\ \hline
        ${N|W}$ &     0 &  1/15 &  -1/15 &  -4/15 &  4/15 \\ \hline
        ${W|S}$ &     0 &  4/15 &   1/15 &  -1/15 & -4/15 \\ \hline
    \end{tabular}
    \begin{align*}
        F_{\edge,\edged} = \alpha_W^{\edged} F_W + \alpha_E^\edged F_E + \alpha_S^\edged F_S + \alpha_N^\edged F_N +  \alpha_B^\edged F_B.
    \end{align*}
\caption{Primal flux over a pyramidal cell and coefficient of the dual fluxes}
\label{tab:pyramid_fluxes}
\end{table}

\medskip
 
Note that the algebraic / least-square technique that we described for general cells may also be applied to the cells for which a geometrical construction is possible; we checked that the coefficients obtained through this technique are the same as those obtained by choosing the momentum field as in Section \ref{sec:geom} and \cite{ans-11-anl}.
\color{black}

\section{Lax-Wendroff consistency of the convection operator}
\label{sec:consis}

The derivation of the discrete convection operator presented in the above section is based on stability arguments only; besides, neither the dual mass fluxes nor the shape of the dual control volumes themselves  seem to be precisely defined, so that the method itself may appear rather puzzling.
It then seems worth assessing its convergence properties.

Let us first mention that recent theoretical results have been already obtained for the type of scheme under study: in \cite{lat-18-conv},  the convection term of the variable density incompressible Navier-Stokes equations is discretized on quadrangles with the above method, and convergence of the approximate solutions is proven; in \cite{bab-14-sta}, a linearized version of this operator is shown to satisfy first-order error estimates.
Last but not least, numerical experiments show that this discretization is in most tests more accurate than standard variants, see e.g. \cite{gas-18-mus,her-18-cons}.
Since we are dealing with the full Euler equations, we do not have enough estimates to ensure the compactness of the approximate solutions, and therefore we only show here the Lax-Wendroff consistency of this operator, that is to say, assuming some bounds and compactness on the sequence of approximate functions associated to a vanishing time and mesh steps, we prove that the weak form of the discrete convection operator tends to the weak form of the continuous convection operator. 
The proof relies on some general results \cite{gal-19-wea,gal-22-lax} which we recall in the appendix.

\medskip

Let $\mesh$ be a given mesh and let $\mathcal T=\{(t_n)_{n \in  \llbracket 0, N \rrbracket},\ t_n= n\,\deltat\}$ be a given time discretization of time step $\deltat$.
We define the discrete function associated to the density and the velocity by:
\begin{equation}
 \label{eqdef:discrete_functions}
    \begin{array}{l}
        \rho(\bfx,t) = \rho_K^{n+1} \mbox{ for } \bfx \in K,\ t \in (t_n,t_{n+1}],\ K \in \mesh,\ n \in \llbracket 0, N-1 \rrbracket,
        \\[1ex]
        \bfu(\bfx,t) = \bfu_\edge^{n+1} \mbox{ for } \bfx \in D_\edge,\ t \in (t_n,t_{n+1}],\ \edge \in \edges,\ n \in \llbracket 0, N-1 \rrbracket,
    \end{array}
\end{equation}
and the interpolate of a test function $\varphi$ by 
\begin{multline} \label{eqdef:interp_phi}
    \interpm(\varphi)(\bfx,t) =\varphi_{\edge}^n \mbox{ for }\bfx \in D_\edge \mbox{ and } t \in (t_n, t_{n +1}),
    \\
    \mbox{with }\varphi_{\edge}^n = \frac 1 {|D_\edge|} \int_{D_\edge} \varphi(\bfx,t_n)\dx \quad \mbox{for } \edge \in \edges \mbox{ and } n \in \llbracket 0, N-1 \rrbracket.
    \qquad 
\end{multline}
In the sequel, we also use the notation $\varphi_K^n$ defined by:
\begin{equation}
    \label{eqdef:phiKn}
    \varphi_{K}^n = \frac 1 {|K|} \int_{K} \varphi(\bfx,t_n)\dx \qquad \mbox{for } K \in \mesh \mbox{ and } n \in \llbracket 0, N-1 \rrbracket.
\end{equation}
The regularity of a mesh $\mesh$ is measured by the parameters $\theta_1(\mesh)$ and $\theta_2(\mesh)$ defined by
\begin{subequations}
\begin{align}\label{eqdef:theta-un-deux}
    \theta_1(\mesh) & = \max_{K \in \mesh} \frac{\diam(K)^d}{|K|}, \\
    \theta_2(\mesh) & = \max \Bigl\{ \frac{|K|}{|L|},\ K \mbox{ and } L \mbox{ adjacent cells of } \mesh \Bigr\}.
\end{align}
\end{subequations}

For $i\in \llbracket 1,d \rrbracket$, consider the (scalar) convection operator defined by:
\begin{align}
    \mathcal C_i(\rho,\bfu) : 
    & \quad
    \Omega\times(0,T) 	\to \xR,\nonumber
    \\ & \quad \label{eqdef:Ci}
    (\bfx,t) \mapsto
    \mathcal C_{i,\edge}^n(\rho,\bfu)_\edge, \mbox{ for } \begin{cases}
        \bfx \in D_\edge,\ \edge \in \edges, \\
        t \in (t_n,t_{n+1}),\ n \in \llbracket 0, N-1 \rrbracket,
    \end{cases}
\end{align}
where, for $\edge \in \edges$ and $n \in \llbracket 0, N-1 \rrbracket$,
\begin{align}
     \mathcal C_{i,\edge}^n(\rho,\bfu)_\edge & =  \frac{1}{\deltat}\! \bigl(\rho_\Ds^{n+1} u_{i,\edge}^{n+1} - \rho_\Ds^{n} u_{i,\edge}^n\bigr) + \dive_{\edge}(\rho^n u_i^n \bfu^n),    \label{eqdef:Ciedge}
\end{align}
with $\dive_{\edge}(\rho^n u_i^n \bfu^n)$ defined by \eqref{eq:conv_vel}. 

\begin{theorem}[Lax-Wendroff consistency of the convection operator] \label{theo:cons}
 Let $(\mesh\exm)_\mnn$ and $(\mathcal T\exm)_\mnn$ be a sequence of space-time discretisations be given, with $h_{\mesh\exm}$ and $\deltat\exm$ tending to zero, and let $(\rho\exm, \bfu\exm)_\mnn$ be an associated sequence of discrete functions.
    We suppose that
    \begin{align}\label{eq:hyp_reg}
        &\exists\  \theta \in \xR \mbox{ such that } \max \{\theta_1(\mesh\exm),\ \theta_2(\mesh\exm),\ m \in \xN\} \leq   \theta,\\
        &\label{eq:hyp_faces}
        \exists\ N_\edges \in \xR \mbox{ such that } \max \{\mathrm{card}(\edges(K)),\ K\in \mesh\exm,\ m \in \xN\} \leq  N_\edges.
    \end{align}
    We also suppose that there exist $\ctel{c_rho}, \ctel{c_u} \in \xR_+$ independent of $m$,  and that there exist $\bar \rho \in L^\infty(\Omega \times [0,T))$ and $\bar \bfu \in L^\infty(\Omega \times (0,T))^d$, such that 
    \begin{align}
      &\Vert\rho\exm\Vert_{L^\infty(\Omega \times (0,T))} \le \cter{c_rho} \; \forall m \in \xN, \label{hyp:bound_rho} \\
      &\Vert\bfu\exm\Vert_{L^\infty(\Omega \times (0,T))^d} \le \cter{c_u} \; \forall m \in \xN, \;  \label{hyp:bound_u} \\
      &\Vert \rho\exm -\bar \rho \Vert_{L^\infty(\Omega \times (0,T))} \to 0 \mbox{ as } m \to +\infty,
      \label{hyp:conv_rho}\\
      &\Vert\bfu\exm -\bar \bfu \Vert_{L^\infty(\Omega \times (0,T))^d} \to 0 \mbox{ as } m \to +\infty. \label{hyp:conv_u} 
    \end{align}
For $i \in \llbracket,\rrbracket$, let $C_i\exm(\rho\exm,\bfu\exm)$ be the convection operator defined by \eqref{eqdef:Ci}-\eqref{eqdef:Ciedge} for each mesh $\mesh\exm$. 
Then, for any function $\varphi \in C^\infty_c(\Omega \times (0,T))$,
    \begin{multline} \label{lw_conv}
        \lim_{m \to +\infty}\int_0^T \int_\Omega \mathcal C_i\exm(\rho\exm,\bfu\exm)\ \interpm(\varphi) \dx \dt \\
        = - \int_\Omega  \rho(\bfx,0) \,u_i(\bfx,0)  \varphi (\bfx,0)\dx
        - \int_0^T \int_\Omega \bigl(\rho\,u_i\ \partial_t \varphi + \ \rho\,u_i\,\bfu \cdot \gradi \varphi\bigr) \dx \dt.
    \end{multline}
\end{theorem}

\begin{proof}
    Let $\varphi \in C^\infty_c(\Omega \times (0,T))$ be given.
    By definition of the convection operator, we have:
    \[
        \int_0^T \int_\Omega \mathcal C_i\exm(\rho\exm,\bfu\exm)\ \interpm(\varphi) \dx \dt = T\exm + X\exm
    \]
    with
    \[
        \begin{array}{l} \displaystyle
        T\exm = \sum_{n=0}^{N\exm-1} \sum_{\edge \in \edges\exm} |D_\edge|\ \varphi_\edge^{n}\ \Bigl((\rho\exm)_\Ds^{n+1}\, (u_i\exm)_\edge^{n+1} - (\rho\exm)_\Ds^{n}\, (u_i\exm)_\edge^n \Bigr),
        \\[3ex] \displaystyle
        X\exm = \sum_{n=0}^{N\exm-1} \deltat\exm\ (X\exm)^{n+1},
        
        \end{array}
    \]
    where
    \[
    (X\exm)^n=\sum_{\edge \in \edges\exm} \varphi_\edge^{n} \sum_{\edged \in \edgesd(D_\edge)} F_{\edge,\edged}^n (\rho\exm,\bfu\exm)\ (u_i\exm)_\edged^{n }.
    \]
    Throughout this proof, we suppose that the space and time steps are small enough for the test function $\varphi$ to vanish in the boundary cells and at the last time step.
    Hence, "boundary terms" appear neither in time nor space discrete integration by parts.
    
    \medskip
    The convergence of the time derivative term $T\exm$ is obtained by applying \cite[Lemma 2.7]{gal-22-lax}, which we recall in the appendix with weaker assumptions that are sufficient in the present work, see Lemma \ref{lem:time-cons}.
    In this latter lemma, we choose $\mathcal P\exm$ to be the (virtual) dual mesh defined by the dual cells $D_\edge$ for $\edge \in \edgesint$, and for each component of the velocity, $i = 1, \ldots d$, the function $\beta$ is defined by $\beta(\rho,u_i) = \rho u_i$; 
    Let us first check that the condition \eqref{hyp:condi} of  Lemma \ref{lem:time-cons} is satisfied \ie~ that
    \begin{equation}\label{hyp:rho-u_zero}
       \sum_{\edge\in\edgesint}\int_{D_\edge} \left| (\rho^0 u^0_i)_{D_\edge} - \rho_0(\bfx) u_{0,i}(\bfx) \right| \dx \to 0 \mbox{ as } m \to + \infty.
    \end{equation}
    By \eqref{scheme:init}, $(\rho^0 u^0_i)_{D_\edge} =  \dfrac 1 {|D_{\sigma}|} \left(|D_{K,\sigma}| \rho^0_K + |D_{L,\sigma}| \rho^0_L\right) (u_i^0)_\edge$.
    Now, remark that 
    \[
    \dfrac 1 {|D_{\sigma}|} \left(|D_{K,\sigma}| \rho_0(\bfx) + |D_{L,\sigma}|\rho_0(\bfx)\right)  u_i^0(\bfx)= \rho_0(\bfx) u_{0,i}(\bfx);
    \]
    moreover the piecewise functions $\rho^0$ and $u_i^0$ defined by \eqref{scheme:init} converge respectively to $\rho_0$ and $u_{0,i}$ in $L^p(0,T;\Omega)$ for any $p\ge 1$ so that the piecewise constant function $(\rho u_i)^0= \sum_{\edge\in\edgesint} (\rho^0 u_i^0)_{\Ds} \characteristic_{\Ds} $ converges to $\rho_0 u_{0,i}$ in $L^p(0,T;\Omega)$ for any $p \in [1,+\infty[$. 
    Therefore \eqref{hyp:rho-u_zero} holds. 
    
    \smallskip
    
    Let us then show that the assumption \eqref{hyp:t} of Lemma \ref{lem:time-cons} holds; in the present context, it reads:
      \begin{equation}\label{hyp:rho-u_time}
       R_{t}\exm = \sum_{n=0}^{N\exm}\!\!\! \deltat\exm\!\!\!\! \sum_{\substack{\edge\in\edgesint\exm\\ \edge=K|L}} \!\!\! \int_\Ds \Bigl|  (\rho u_i)_\Ds^n   -  \rho^{n}(\bfx) u_i^n(\bfx) \Bigr| \dx \to 0 \mbox{ as } m \to + \infty.  
    \end{equation}
    By  the definition \eqref{eq:pd2} of $\rho_\Ds^n$, we get that 
    \[
      R_{t}\exm = \sum_{n=0}^{N\exm} \deltat\exm \sum_{\substack{\edge\in\edgesint\exm\\ \edge=K|L}} I_\edge 
    \]
    with 
    \[
    I_\edge = \int_\Ds \Bigl|\bigl( \dfrac 1 {|D_{\edge}|}\left(|D_{K,\edge}| \rho_K + |D_{L,\edge}| \rho_L\right)u_{i,\edge}^n - \rho^{n}(\bfx) u_i^n(\bfx) \Bigr| \dx.
    \]
Let us decompose $I_\edge = I_{K,\edge} +  I_{L,\edge} $ with 
    \begin{align*}
       I_{K,\edge} &=  |D_{K,\edge}| \Bigl|\bigl( \dfrac 1 {|D_{\edge}|}\left(|D_{K,\edge}| \rho_K^n + |D_{L,\edge}| \rho_L^n\right)u_{i,\edge}^n - \rho_K^{n}  u_{i,\edge}^n \Bigr|  \\
      & = \dfrac{|D_{K,\edge}|}{|D_{\edge}|} \Bigl|\bigl( \left(|D_{K,\edge}| \rho_K^n + |D_{L,\edge}| \rho_L^n\right)u_{i,\edge}^n - |D_{\edge}| \rho_K^{n} 
      u_{i,\edge}^n \Bigr| \\
      & = \dfrac{|D_{K,\edge}||D_{L,\edge}|}{|D_{\edge}|} |\rho_K^n -\rho_L^n |u_{i,\edge}^n.´
    \end{align*}
 From the bound \ref{hyp:bound_u}, we thus get that  
    \[
        R_t\exm \leq 
       2\sum_{n=0}^{N\exm} \deltat\exm \sum_{\substack{\edge\in\edgesint\exm\\ \edge=K|L}} \dfrac{|D_{K,\edge}||D_{L,\edge}|}{|D_{\edge}|} |\rho_K^n -\rho_L^n|.
    \]
   The convergence to zero of $R_{t}\exm$ follows thanks to Theorem \ref{thm:trans_bound}. 
   Indeed, taking $\mathcal{P} = \mesh\exm$, $P= K$, $Q=L$, it is clear that $\theta_{\mathcal P} = \max_{\substack{\edge \in \edgesint\exm\\ \edge=K|L}} \dfrac{|D_{K,\edge}|} {|K|} \dfrac { |D_{L,\edge}|}{|D_{\edge}|} \le 1$.
   
   \smallskip
   
   We have thus proven that  the assumptions of Lemma \ref{lem:time-cons} hold, and so that:
   \begin{equation} \label{lim-time}
     \lim_{m\to +\infty} T\exm  = - \int_\Omega  \rho(\bfx,0) \,u_i(\bfx,0)  \varphi (\bfx,0)\dx  - \int_0^T \int_\Omega  \rho\,u_i\ \partial_t \varphi \dx \dt.
    \end{equation}

 \medskip

    Let us now turn to the convection term.
    The first step in the analysis of this term consists in writing it as the sum of a remainder and a second term involving only the primal mass fluxes instead of the dual ones.
    Dropping  the dependency on $m$ at the right hand-side and decomposing the sum yields:
    \begin{align}
        (X\exm)^{n} & =\sum_{\edge \in \edges\exm} \varphi_\edge^{n} \sum_{\edged \in \edgesd(D_\edge)} F_{\edge,\edged}^n\ u_{i,\edged}^{n} \nonumber \\
        \label{eq:Tc}
        & = \sum_{K \in \mesh\exm}\ \sum_{\edge \in \edges(K)} \varphi_\edge^{n} \sum_{\substack{\edged \in \edgesd(D_\edge),\\ \edged \subset K}} F_{\edge,\edged}^n\ u_{i,\edged}^{n}.
    \end{align}
    For $K \in \mesh$ and $\edge \in \edges(K)$, let us recast the mass balance over the half-diamond cells \eqref{eq:F_syst} as:
    \begin{equation}\label{eq:mss_DKS}
        \sum_{\edged \in \edgesd(D_\sigma),\ \edged \subset K} F_{\edge,\edged}^n= -F_{K,\edge}^n + \xi_K \ (\eth \rho)_K^n
        \qquad \mbox{with } (\eth \rho)_K^n = \sum_{\edge' \in \edges(K)} F_{K,\edge'}.
    \end{equation}
    We recall that the coefficient $\xi_K$ does not depend on the face of $K$ (hence the suppression of the index $\edge$ in the notation) and only depends on the geometry of $K$ (so, for instance, for a two-dimensional mesh of quadrangles, $\xi_K=1/4$ for all the cells of the mesh).
    We now remark that, thanks to this relation,
    \begin{align*}
        \sum_{K \in \mesh\exm}\ \sum_{\edge \in \edges(K)} \varphi_\edge^{n} u_{i,\edge}^{n} \sum_{\substack{\edged \in \edgesd(D_\edge),\\ \edged \subset K}} F_{\edge,\edged}^n
        = -\sum_{K \in \mesh\exm}\ \sum_{\edge \in \edges(K)} \varphi_\edge^{n} u_{i,\edge}^{n} F_{K,\edge}^n 
        + \sum_{K \in \mesh\exm} \xi_K \ (\eth \rho)_K^n \sum_{\edge \in \edges(K)} \varphi_\edge^{n} u_{i,\edge}^{n}.
    \end{align*}
    By conservativity, the first sum at the right hand-side of this relation vanishes.
    Thanks to this equation, we recast Equation \eqref{eq:Tc} as $(X\exm)^{n}=(X_1\exm)^n+(X_2\exm)^n$ with
    \begin{equation}\label{eq:Tc1}
        \begin{array}{l} \displaystyle
            (X_1\exm)^n=  \sum_{K \in \mesh\exm}\ \sum_{\edge \in \edges(K)} \varphi_\edge^{n} \sum_{\substack{\edged \in \edgesd(D_\edge),\\ \edged \subset K}} F_{\edge,\edged}^n\ \bigl( u_{i,\edged}^{n} - u_{i,\edge}^{n}),
            \\[4ex] \displaystyle
            (X_2\exm)^n = \sum_{K \in \mesh\exm} \xi_K \ (\eth \rho)_K^n \sum_{\edge \in \edges(K)} \varphi_\edge^{n} u_{i,\edge}^{n}.
        \end{array}
    \end{equation}
    Let us now decompose $(X_1\exm)^n$ as $(X_1\exm)^n = (X_3\exm)^n + (R_{1}\exm)^n$ with
    \begin{equation}\label{eq:Tc3}
        \begin{array}{l} \displaystyle
            (X_3\exm)^n=  \sum_{K \in \mesh\exm} \varphi_K^{n} \sum_{\edge \in \edges(K)}  \sum_{\substack{\edged \in \edgesd(D_\edge),\\ \edged \subset K}} F_{\edge,\edged}^n\ \bigl( u_{i,\edged}^{n} - u_{i,\edge}^{n}),
            \\[4ex] \displaystyle
            (R_{1}\exm)^n = \sum_{K \in \mesh\exm}\ \sum_{\edge \in \edges(K)} (\varphi_\edge^{n}-\varphi_K^{n}) \sum_{\substack{\edged \in \edgesd(D_\edge),\\ \edged \subset K}} F_{\edge,\edged}^n\ \bigl( u_{i,\edged}^{n} - u_{i,\edge}^{n}).
        \end{array}
    \end{equation}
    Let us show that 
    \begin{equation}
     \label{eq:R1}
      \sum_{n=0}^{N\exm-1} \deltat\exm(R_{1}\exm)^n \to 0 \mbox{ as } m\to+\infty.
    \end{equation}
    Firstly, by the mean value theorem, there exists $\ctel{cphi}$ depending only on $\varphi$ such that 
    \[
     |\varphi_\edge^{n}-\varphi_K^{n}| \le \cter{cphi} (\diam(K) + \diam(L)).
    \]
    Secondly, thanks to the definition \eqref{eqdef:theta-un-deux}, one has
    \[\diam(L)^d \le \theta_1\exm |L| \le \theta_1\exm \theta_2\exm |K| \le \theta_1\exm \theta_2\exm \diam(K)^d,
     \]
    so that, thanks to the regularity \eqref{eq:hyp_reg} of the mesh, 
    \[
    \diam(K)+\diam(L) \leq (1+\theta^{2/d}) \diam(K), 
    \]
and therefore
    \begin{equation}
    \label{eq:bound-phi}
     |\varphi_\edge^{n}-\varphi_K^{n}| \le \cter{cphi} (1+\theta^{2/d}) \diam(K).
    \end{equation}
    By the definition \eqref{eqdef:dualflux} of the dual flux and owing to the  $L^\infty$ estimates \eqref{hyp:bound_rho} and \eqref{hyp:bound_u}, the following bound holds:
    \begin{align*}
        F_{\edge,\edged} = \sum_{\edge' \in \edges(K)} \alpha_K^\edged F_{K,\edge'}\leq \cter{c_rho}\cter{c_u} \diam(K)^{d-1} ,
        \edge \in \edges(K) ,\ \edged \in \edgesd(D_{\edge}) \mbox{ and } \edged \subset K,
    \end{align*}
    Thirdly, for $\edged=\edge|\edge'$, $u_{i,\edged}^n$ is a convex combination of $u_{i,\edge}^n$ and $u_{i,\edge'}^n$.
    These three arguments together yield:
    \begin{align*}
        |(R_{1}\exm)^n| &\leq  \cter{c_rho}\cter{c_u} \cter{cphi} (1+\theta^{2/d}) \sum_{K \in \mesh} \diam(K)^d \sum_{\edge \in \edges(K)} \sum_{\substack{\edged \in \edgesd \exm, \edged \subset K\\ \edged = \edge|\edge'}} |\bfu^n_\edge - \bfu^n_{\edge'}|\\
        &\leq 3 \cter{c_rho}\cter{c_u} \cter{cphi} (1+\theta^{2/d})  \sum_{\substack{\edged \in \edgesd \exm \\ \edged = \edge|\edge'\subset K}} \diam(K)^d |\bfu^n_\edge - \bfu^n_{\edge'}|.
    \end{align*}

    The assertion \eqref{eq:R1} then follows from Theorem \ref{thm:trans_bound} given in the appendix.
    Indeed, we consider for $\mathcal P$ the mesh which consists of the cells $P_\edged= D_{K,\edge} \cup D_{K,\edge'}, \edged = \edge|\edge'$ where $K \in \mesh$ is such that $\edged \subset K$. 
    Then  Theorem \ref{thm:trans_bound}  holds provided that $ \omega_{\edge,\edge'} = \dfrac{\diam(K)^d}{|D_{K,\edge}|+|D_{K,\edge'}|}$ is bounded independently of $m$; this is indeed true, since, thanks to the assumptions \eqref{equal-volume}, \eqref{eq:hyp_reg} and \eqref{eq:hyp_faces},
    \[
     \dfrac{\diam(K)^d}{|D_{K,\edge}|+|D_{K,\edge'}|}= \mathrm{card}\edges(K) \dfrac{\diam(K)^d}{2 |K|} \le\dfrac {\theta N_\edges} 2.
    \]

  \medskip  
    
  Developing the term $(X_3\exm)^n$, we obtain
  \begin{align*}
      (X_3\exm)^n = \sum_{K \in \mesh\exm} \varphi_K^{n} \sum_{\edge \in \edges(K)}  \sum_{\substack{\edged \in \edgesd(D_\edge),\\ \edged \subset K}} F_{\edge,\edged}^n\ u_{i,\edged}^{n}
      - \sum_{K \in \mesh\exm} \varphi_K^{n} \sum_{\edge \in \edges(K)} u_{i,\edge}^{n} \sum_{\substack{\edged \in \edgesd(D_\edge),\\ \edged \subset K}}F_{\edge,\edged}^n.
  \end{align*}
    By conservativity, the first sum at the right hand-side vanishes.
    Using once again \eqref{eq:mss_DKS}, we then obtain $(X_3\exm)^n=(\tilde X\exm)^n+(X_4\exm)^n$
    \begin{equation}
        \begin{array}{l} \displaystyle
            (\tilde X\exm)^n=  \sum_{K \in \mesh\exm} \varphi_K^{n} \sum_{\edge \in \edges(K)} F_{K,\edge}^n u_{i,\edge}^{n},
            \\[4ex] \displaystyle
            (X_4\exm)^n = - \sum_{K \in \mesh\exm} \xi_K \ (\eth \rho)_K^n\ \varphi_K^{n} \sum_{\edge \in \edges(K)} u_{i,\edge}^{n}.
        \end{array}
    \end{equation}
    We gather $(X_4\exm)^n$ with $(X_2\exm)^n$ to obtain yet another remainder term:
    \begin{equation}\label{eq:R2}
        (R_{2}\exm)^{n} =  (X_2\exm)^{n} + (X_4\exm)^{n} = \sum_{K \in \mesh\exm} \xi_K \ (\eth \rho)_K^n \sum_{\edge \in \edges(K)} (\varphi_\edge^{n} - \varphi_K^{n})\ u_{i,\edge}^{n}.
    \end{equation}
    Let us show that $(R_{2}\exm)^{n}$ is indeed a remainder term, in the sense that
    \begin{equation}
     \label{eq:R2to0}
      \sum_{n=0}^{N\exm-1} \deltat\exm(R_{2}\exm)^{n} \to 0 \mbox{ as } m\to+\infty.
    \end{equation}
     For a given cell $K$, let $\bfu_K^n$ be the mean value of the velocities $\bfu_\edge^n$ at the faces of $K$; since $  \sum_{\edge \in \edges(K)} |\edge|\ \rho_K^n\ \bfu_K^n \cdot \bfn_{K,\edge} = 0$, we get 
    \[
        (\eth \rho)_K^n = \sum_{\edge \in \edges(K)} |\edge|\ \rho_\edge^n\ \bfu_\edge^n \cdot \bfn_{K,\edge} = 
        \sum_{\edge \in \edges(K)} |\edge|\ (\rho_\edge^n\ \bfu_\edge^n - \rho_K^n\ \bfu_K^n) \cdot \bfn_{K,\edge}.
    \]
    By the triangle inequality, owing to the the $L^\infty$ estimates \eqref{hyp:bound_rho} and \eqref{hyp:bound_u}, using again \eqref{eq:bound-phi}, we get that there exists $\ctel{c-r2}$ independent of $m$ such that:
    \begin{align*}
       (R_{2}\exm)^n \leq \cter{c-r2} (1+\theta^{2/d}) \sum_{K \in \mesh} \diam(K)^d \sum_{\edge \in \edges(K)} \left( |\rho_\edge^n - \rho_K^n| + |\bfu_\edge^n - \bfu_K^n|\right).
    \end{align*}
    Since for $\edge = K|L$, $\rho_\edge^n$ is a convex combination of $\rho_K^n$ and $\rho_L^n$, and since $\bfu^n_K = \frac{1}{\mathrm{card}(\edges(K))}\sum_{\edge' \in \edges(K)} \bfu^n_{\edge'}$, we obtain:
    \begin{align*}
        (R_{2}\exm)^n \leq C \theta (1+\theta^{2/d}) \Biggr[ \sum_{\substack{\edge \in \edgesint\\ \edge = K|L}} (|K|+|L|)|\rho_K^n - \rho_L^n| + 
        \sum_{\edge \in \edgesint} |D_\edge| \sum_{\substack{\edged \in \edgesd \exm\\ \edged = \edge|\edge'}} |\bfu^n_\edge - \bfu^n_{\edge'}| \Biggr].
    \end{align*}
    Invoking once again Theorem \ref{thm:trans_bound} with the primal mesh  $\mathcal \mesh\exm$ for the first sum of the right hand-side and  the dual mesh for the second one,  yields that \eqref{eq:R2to0} holds.
    
    \medskip
    
    Let us finally prove the convergence of the term $(\tilde X\exm)^n$.
    This sum is the weak form of the divergence part of a new convection operator, posed on primal cells, and defined by:
    \begin{align*}
        \mathcal C_\mesh(\rho,\bfu) : 
        & \quad
        \Omega\times(0,T) 	\to \xR,
        \\ & \quad 
        (\bfx,t) \mapsto
        \mathcal C_K^n(\rho,\bfu), \mbox{ for } \begin{cases}
            \bfx \in K,\ K \in \mesh, \\
            t \in (t_n,t_{n+1}),\ n \in \llbracket 0, N-1 \rrbracket,                       
        \end{cases}
    \end{align*}
    where, for $K \in \mesh$ and $n \in \llbracket 0, N-1 \rrbracket$,
    \begin{align*}
        C_K^n(\rho,\bfu)= \sum_{\edge \in \edges(K)} F_{K,\edge}^n\ u_{i,\edge}^{n}.
    \end{align*}
    Invoking Lemma \ref{lem:space-cons}, we thus have to check the consistency of the flux defined by $G_{K,\edge}=F_{K,\edge}^n\ u_{i,\edge}^{n}$ for $K \in \mesh$ and $\edge \in \edges(K)$, that is to say  assumption \eqref{hyp:x}.
    For $\bfx \in K$, the approximate density is $\rho(\bfx) = \rho_K$ and the approximate velocity $\bfu = \sum_{\edge'\in \edges(K)} \bfu_{\edge'} \characteristic_{D_{K,\edge'}}$. 
    Since  $|K| = {\mathrm{card}\bigl(\edges(K)\bigr)}|D_{K,\edge'}|$ for any $\edge' \in \edges(K)$, the left hand-side of assertion \eqref{hyp:x} reads, for the operator at hand:
    \[
        R_{\dive}\exm=\sum_{n=0}^{N\exm-1} \deltat\exm \sum_{K \in \mesh\exm} \diam(K)
        \sum_{\edge \in \edgesint(K)} R_{K,\edge}^{n},
    \]
    with
    \[
        R_{K,\edge}^{n} = |\edge|\ \left| \bfn_{K,\edge} \cdot 
        \frac 1 {\mathrm{card}\bigl(\edges(K)\bigr)}\ \sum_{\edge' \in \edges(K)} \rho_\edge^n\ u_{i,\edge}^{n}\ \bfu_\edge^n  - \rho_K^{n}\ u_{i,\edge'}^{n}\ \bfu_{\edge'}^{n}
        \right|.
    \]
    In this relation, $\rho_\edge^n$ stands for the approximation of the density at the face $\edge$, and is a convex approximation of $\rho_K^n$ and $\rho_L^n$ where $K$ and $L$ are the cells separated by $\edge$.
     The proof that $R_{\dive}\exm$ tends to zero relies on Theorem \ref{thm:trans_bound}.
    So, we have to recast this term as a collection of jumps, in time or in space, and show that the weights of these jumps are such that Theorem \ref{thm:trans_bound} applies.
    By the triangle inequality and owing to the bounds \eqref{hyp:bound_rho} and \eqref{hyp:bound_u}, we obtain 
     \begin{equation} \label{eq:R_div}
     R_{K,\edge}^{n} 
         \leq |\edge|\ \left[ \cter{c_u}^2 | \rho_\edge^{n} - \rho_K^{n}| +   2 \cter{c_u}\cter{c_rho} |\bfu_{\edge}^{n}\   -  \bfu_{\edge'}^{n}| \right]
     \end{equation}
     Hence $R_{\dive}\exm \le T_{\rho}\exm + T_{\bfu}\exm$ with 
     \begin{align}
      &  T_{\rho}\exm \le \cter{c_u}^2 \sum_{n=0}^{N\exm-1} \deltat\exm \sum_{K \in \mesh\exm} \diam(K) \sum_{\edge \in \edgesint(K)} |\edge|\   |\rho_\edge^{n} - \rho_K^{n}| \label{Trho}\\
      &  T_{\bfu}\exm \le 2 \cter{c_u}\cter{c_rho}  \sum_{n=0}^{N\exm-1} \deltat\exm \sum_{K \in \mesh\exm} \diam(K)  \sum_{\edge \in \edgesint(K)} |\edge| |\bfu_{\edge}^{n}\   -  \bfu_{\edge'}^{n}|.\label{Tu}
     \end{align}
    Remarking that for $\edge = K|L$, $\rho_\edge^{n}$ is a convex combination of $\rho_K^{n}$ and $\rho_L^{n}$, it follows that
    \[
         T_{\rho}\exm \leq  \cter{c_u}^2  \sum_{n=0}^{N\exm-1} \deltat\exm \sum_{K \in \mesh\exm} \diam(K) \sum_{\substack{\edge \in \edgesint(K)\\ \edge=K|L}} |\edge|\ |\rho_K^{n} - \rho_L^{n}|,
    \]
    so that, thanks to the regularity of the sequence of meshes,
    \[
         T_{\rho}\exm  \leq\cter{c_u}^2 \theta \sum_{n=0}^{N\exm-1} \deltat\exm \sum_{\edge \in \edgesint,\ \edge=K|L} (|K| + |L|)\ |\rho_K^{n} - \rho_L^{n}|,
    \]
    which converges to zero thanks to Theorem \ref{thm:trans_bound}, still by regularity of the sequence of meshes.  
    
    Reordering the summation in the bound \eqref{Tu} of $T_{\bfu}\exm$, we get that:
    \begin{align*}
        T_{\bfu}\exm 
        & \displaystyle
        \leq 2 \cter{c_u}\cter{c_rho}   \sum_{n=0}^{N\exm-1} \deltat\exm \sum_{K \in \mesh\exm} \diam(K) \sum_{\edge \in \edges(K)} |\edge|\ \sum_{\edge' \in \edges(K)} |\bfu_\edge^{n+1} - \bfu_{\edge'}^{n+1}|
        \\ & \displaystyle
        \le 2 \cter{c_u}\cter{c_rho}   \sum_{n=0}^{N\exm-1} \deltat\exm \sum_{\substack{\edged \in \edgesd\exm,\ \edged \subset K,\\ \edged=\edge|\edge'}} \diam(K)\ \bigl(|\edge|+|\edge'|\bigr)\ |\bfu_\edge^{n+1} - \bfu_{\edge'}^{n+1}|,
    \end{align*}
    and the convergence to zero of $T_{\bfu}\exm $  again follows by Theorem \ref{thm:trans_bound}, by construction of the dual cells, $|D_{K,\edge}| = |D_{K,\edge'}| =|K|/\mathrm{card}(\edges(K))$ and since $\diam(K)\, (|\edge|+|\edge'|)\leq 2\,\theta\ |K|$ .  
    Hence 
    \begin{align*}
        \lim_{m \to + \infty} (\tilde X\exm)^n = 
        - \int_\Omega  \rho(\bfx,0) \,u_i(\bfx,0)  \varphi (\bfx,0)\dx 
        - \int_0^T \int_\Omega \bigl(\rho\,u_i\ \partial_t \varphi + \ \rho\,u_i\,\bfu \cdot \gradi \varphi\bigr) \dx \dt.
    \end{align*}
    This limit, together with the limits \eqref{eq:R1},\eqref{eq:R2}, concludes the proof of Theorem \ref{theo:cons}.
\end{proof}

\section{Numerical tests}\label{sec:num}
In this section, we present numerical tests to assess the validity of the proposed discretization.
The computations presented here are performed with the open-source CALIF$^3$S software developed at IRSN \cite{califs}.

\subsection{A three dimensional Mach=10 shock on a column}

We consider here the problem of a uniform three-dimensional shock, that is impeded by an obstacle in the form of a circular column.
Indeed, the domain $\Omega$ consists in the $[0,0.4]\times[0,0.41]\times[0,0.4]$ cube, with a cylindrical column of radius $0.1$ and height $0.3$, which has the center of its basis located at the point $(0.2,0.2,0)$.

At the initial time, the flow is supposed at rest, and is initialized with the right state of the problem, that is:
\begin{align*}
    \begin{bmatrix} \bfu_R = (0, 0, 0)^t \\ \rho_R=1.4 \\ p_R = 1 \end{bmatrix}.
\end{align*}
Then, a strong shock is supposed to be coming uniformly from the left side of the domain, with the given profile:
\begin{align*}
    \begin{bmatrix} \bfu_L = (8.25, 0, 0)^t \\ \rho_L= 8 \\ p_L = 116.5 \end{bmatrix}.
\end{align*}
This profile is determined through the Rankine-Hugoniot conditions, to retrieve a velocity of the shock equal to $\omega = 10$.
Moreover, the coefficient $\gamma$ is equal to $\gamma = 1.4$, so the speed of sound in the pre-shock state is equal to $c = \sqrt{\frac{\gamma p}{\rho}} = 1$ (which explains the Mach number $M=\omega / c = 10$).

The other boundaries condition are given as follows: outlet boundaries conditions are imposed on the right side of the domain, whereas all the other boundaries (that is the upper, lower, front and behind faces of the box, as well as the cylinder) are considered as walls with slip boundary conditions.

The problem is tested on prismatic and pyramidal meshes.
The meshes are constructed as follows: first, a quadrilateral mesh of the bottom of the box is built so as to carefully fit the contour of the obstacle; it consists in 30160 quadrilaterals.
A coarse version of this mesh, with and without the obstacle, can be seen in Figure \ref{fig:mesh_obstacle_2D}.
Then, if one wants a prismatic mesh, the quadrangles are split along a diagonal to obtain triangles.
The resulting triangles are then extruded into $123$ layers in the $z$-direction to form prismatic cells.
The pyramidal mesh is obtained by extruding the quadrangles once again in $123$ layers, thus leading to a hexahedra mesh; the hexahedra are then decomposed into six pyramids.
Finally, the mesh is perforated to fit the cylindrical column.
This leads to very fine meshes, composed of 6828960 cells for the prismatic mesh and 20486880 cells for the pyramidal one.
A section of the coarse version of both meshes can be seen in Figure \ref{fig:mesh_obstacle_3D}.

The results are given on Figure \ref{fig:choc_column_prism} and \ref{fig:choc_column_pyr}, in the form of sections of the domain at different heights: one at the bottom of the column, one just on top of the column and one at the top of the domain.
On both meshes, the results are similar.
This has to be expected since both meshes are very refined, so the computation can be considered as converged.
Moreover, the results are conforming to what can be expected. 
Indeed, at the top of the domain, the shock is not perturbed by the column, so the profile is identical to a pure shock problem.
At the foot of the column, the shock bounces on the obstacle, which creates a zone with high density.
This also causes the creation of a Mach stem on the right part of the obstacle.
At the top of the column, the profile is as expected an intermediary step between the two previous profiles.

\begin{figure}[ht!]
    \centering
    \includegraphics[width = 0.2\textwidth]{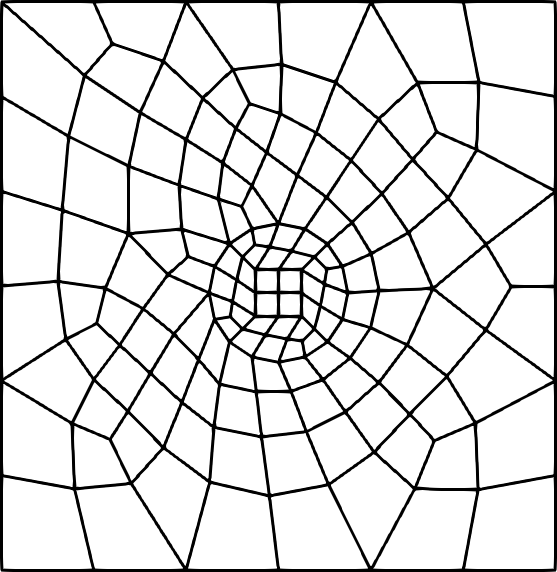}%
    \hspace{1cm}
    \includegraphics[width = 0.2\textwidth]{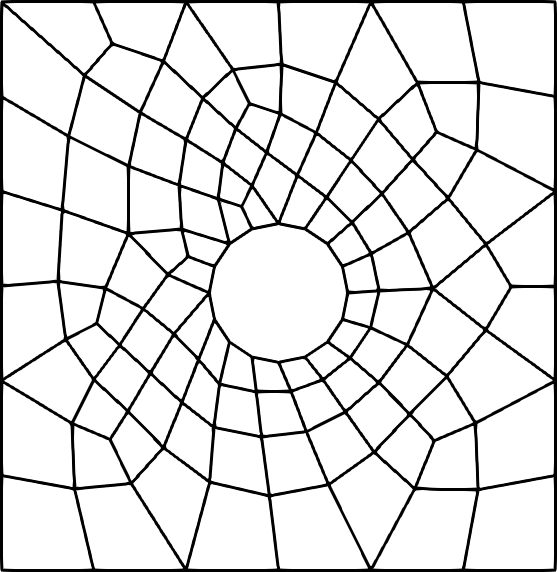}
    \caption{Two dimensional coarse mesh for the shock on a column}
    \label{fig:mesh_obstacle_2D}
\end{figure}

\begin{figure}[ht!]
    \centering
    \includegraphics[width = 0.22\textwidth]{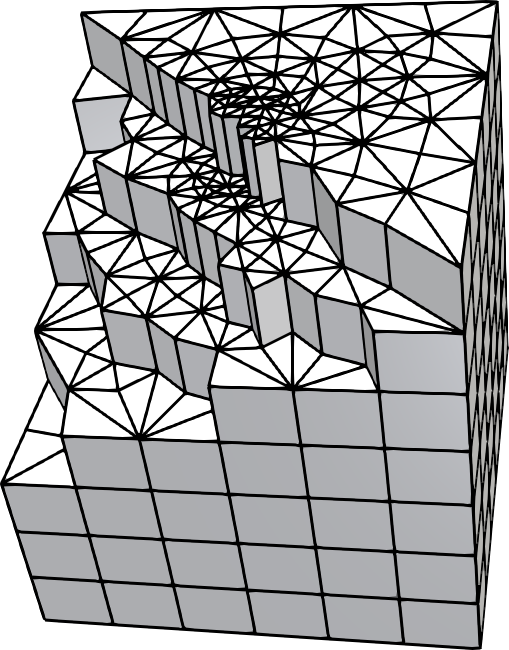}%
    \hspace{1cm}
    \includegraphics[width = 0.22\textwidth]{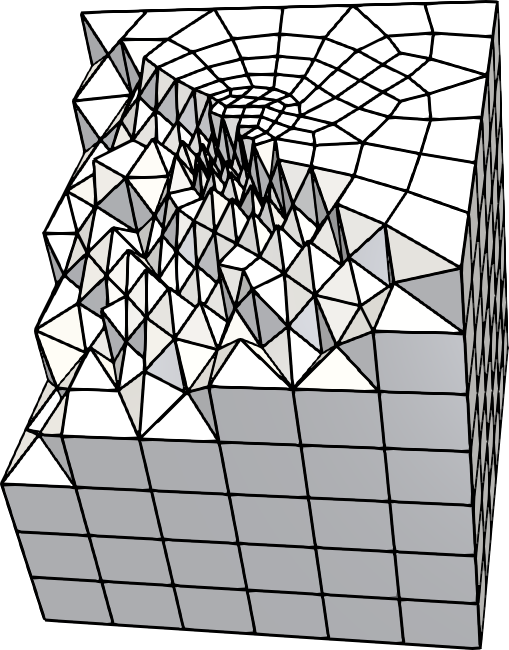}
    \caption{Cut on the above part of the three dimensional coarse meshes for the shock on a column. Left: prismatic mesh. Right: pyramidal mesh.}
    \label{fig:mesh_obstacle_3D}
\end{figure}

\begin{figure}[ht!]
    \centering
    \includegraphics[width = 0.25\textwidth]{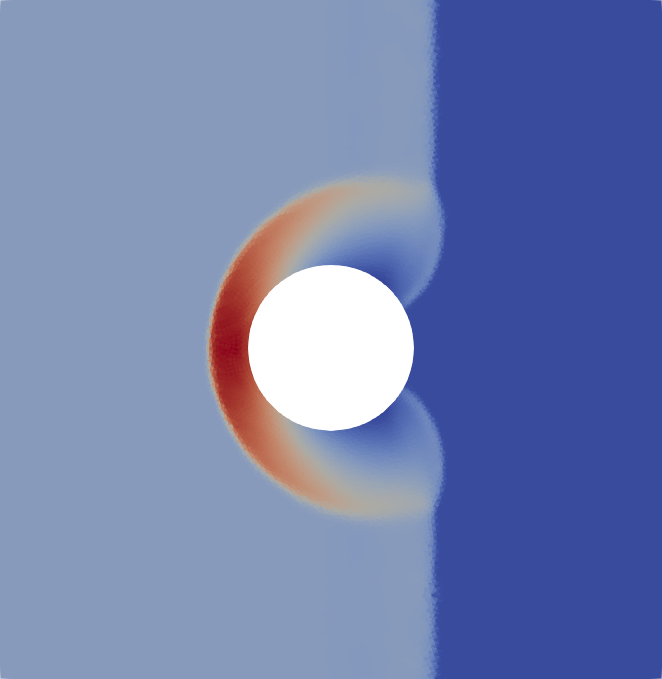}%
    \hspace{1cm}
    \includegraphics[width = 0.25\textwidth]{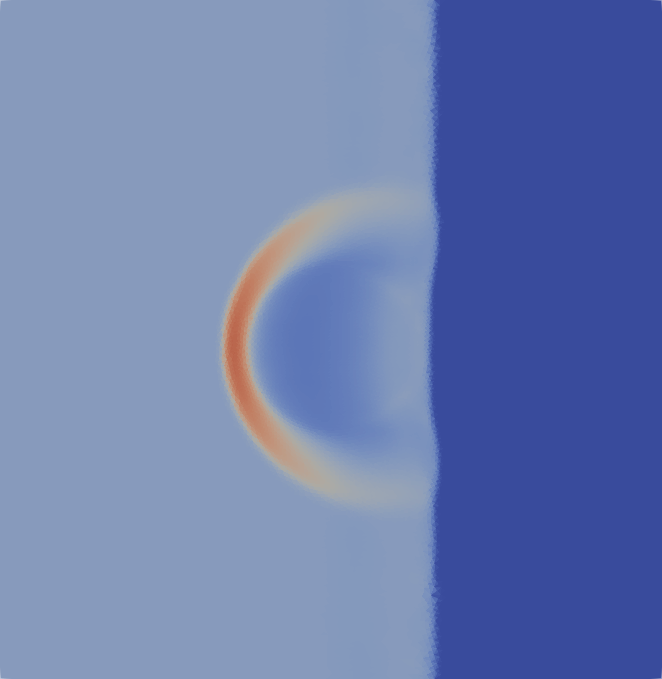}%
    \hspace{1cm}
    \includegraphics[width = 0.25\textwidth]{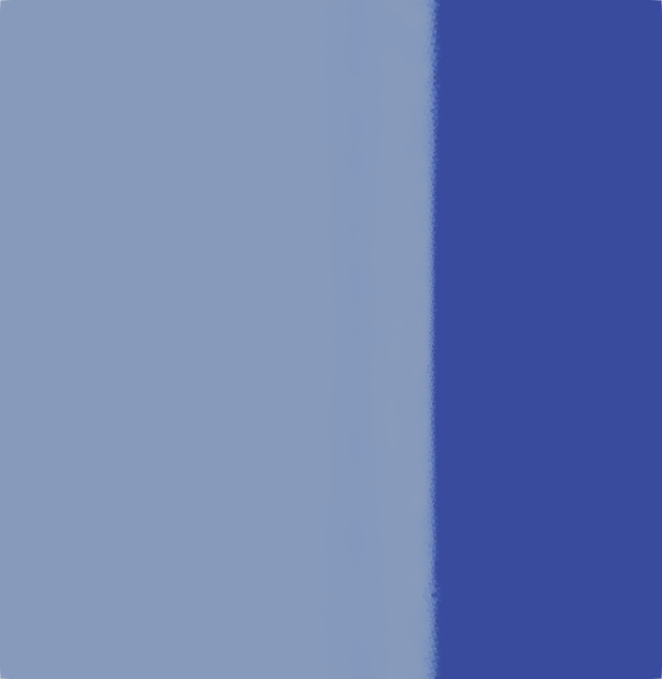}
    \caption{Density for the problem of the shock on a column on a prismatic mesh, at time $t=0.026$. From left to right : cut at $z=0.01$, cut at $z=0.31$ and cut at $z=0.4$.}
    \label{fig:choc_column_prism}
\end{figure}

\begin{figure}[ht!]
    \centering
    \includegraphics[width = 0.25\textwidth]{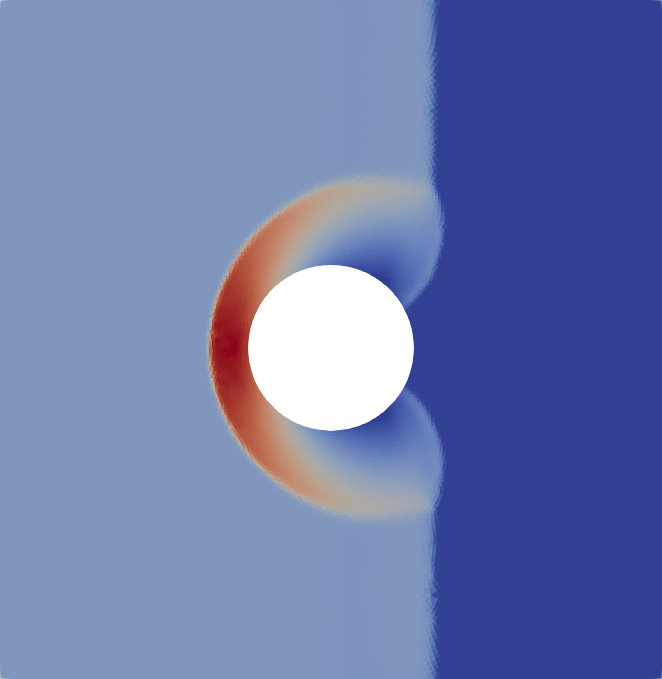}%
    \hspace{1cm}
    \includegraphics[width = 0.25\textwidth]{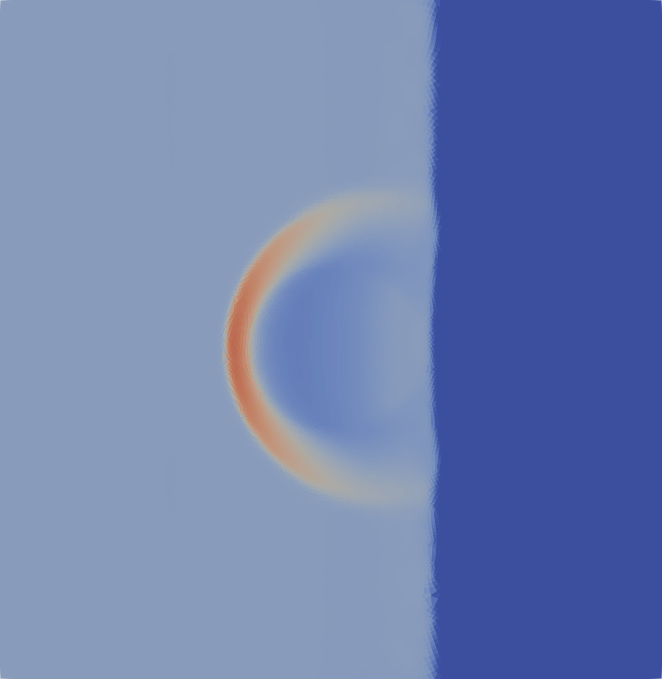}%
    \hspace{1cm}
    \includegraphics[width = 0.25\textwidth]{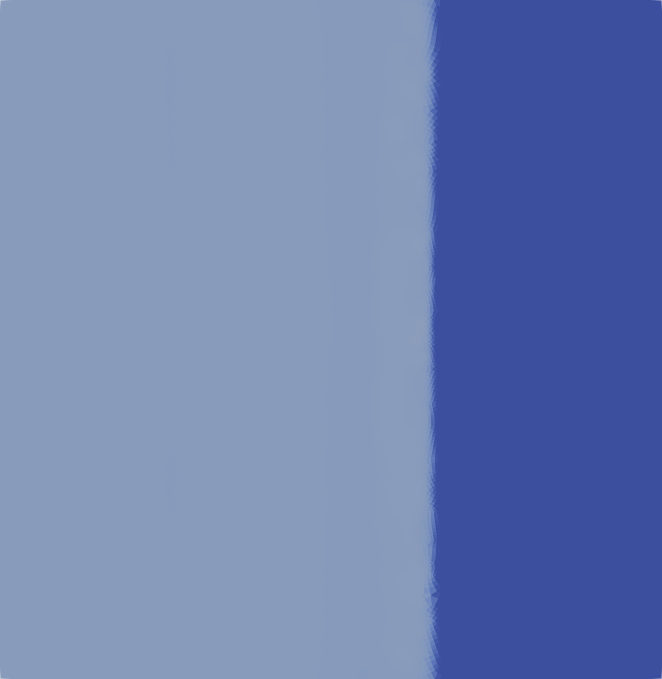}
    \caption{Density for the problem of the shock on a column on a pyramidal mesh, at time $t=0.026$. From left to right : section at $z=0.01$, $z=0.31$ and  $z=0.4$.}
    \label{fig:choc_column_pyr}
\end{figure}

\subsection{A shock with reflexive boundary condition}

We finally turn to a test case where an analytical solution may be obtained through the Rankine-Hugoniot conditions.
We first consider a one dimensional domain $\Omega_{1D}=[0,5]$.
At the initial time $t=0$, and for all $x<2$, the solution is supposed to be "at rest" in the left part of the domain, with a profile given by:
\begin{align}
    \begin{bmatrix}
    u_{L1} = 0 \\ \rho_{L1}=1.292 \\ p_{L1} = 10^5
    \end{bmatrix},
\end{align}
We suppose that a shock is formed with a Mach number equal to $M=10$.
Since the Mach number is defined as $M=\frac{\omega}{c}$, where $c = \sqrt{\frac{\gamma p_{L1}}{\rho_{L1}}}$ is the speed of sound in the pre-shock state (and $\gamma = 1.4$) and $\omega$ is the velocity of the shock, we can determine the right part of the solution for $x\geq 2$, which is given by:
\begin{align}
    \begin{bmatrix}
    u_{R1} = 2 c  \frac{1 - M^2}{M(1+\gamma)}\\ 
    \rho_{R1}=\frac{M^2(1 + \gamma)}{M^2(\gamma-1) +2} \rho_{L1} \\ 
    p_{R1} = \frac{2 \gamma M^2 + 1 - \gamma}{1 + \gamma} p_{L1}
    \end{bmatrix},
\end{align}
The shock then moves from the right to the left; on the left side of the domain, reflexive boundary conditions are imposed, whereas Dirichlet boundary conditions are imposed on the right side of the domain, with the values being fixed by the right state.

For this problem, we can determine the exact solution up to a certain time $T_{\mathrm max}$.
Indeed, up to a time $T_{\mathrm{sym}}=\frac{2}{\omega}$, the exact solution is given by the left state for $x < 2 - \omega t$ and by the right state for $x \geq 2 - \omega t$.
At the time $T_{\mathrm{sym}}$, the shock reflects on the left boundary.
Due to the reflexive boundary conditions, another shock is obtained, for which the velocity on the new left state $u_{L2}$ is equal to zero, whereas the right state is given as previously, that is:
\begin{align}
    \begin{bmatrix}
    u_{R2} = u_{R1}\\ 
    \rho_{R2} = \rho_{R1} \\ 
    p_{R2} = p_{R1}
    \end{bmatrix},
\end{align}
Using the Rankine-Hugoniot condition, one may determine $\omega_2$ the velocity in the newly formed shocked, that is:
\begin{align*}
    \omega_2 = u_{R1} \frac{3-\gamma}{4} + \frac{1}{2} \sqrt{ \frac{(u_{R1}(\gamma+1))^2}{4}+ 4 \frac{\gamma p_{R1}}{\rho_{R1}}}
\end{align*}
and the new left state, given by:
\begin{align}
    \begin{bmatrix}
    u_{L2} = 0\\ 
    \rho_{L2} = \frac{\rho_{R1}(\omega_2 - u_{R1})}{\omega_2} \\ 
    p_{L2} = \rho_{R1} u_{R1}( u_{R1} - \omega_2) + p_{R1}
    \end{bmatrix},
\end{align}
Then, for $t \in [T_{\mathrm{sym}}, T_{\mathrm max}]$, where $T_{\mathrm max}$ is the time at which the shock reaches the right boundary, \ie $\, T_{\mathrm max} = \frac{5}{\omega_2} + T_{\mathrm{sym}}$, the solution is given by the left state $L2$ for $x < \omega_2( t - T_{\mathrm{sym}})$ and by the right state $ x \geq \omega_2 ( t - T_{\mathrm{sym}} )$.

To determine the accuracy of the scheme using three-dimensional cells, the domain $\Omega_{1D}$ is enhanced for the numerical tests into a three-dimensional domain, given by $\Omega = \Omega_{1D} \times [0, 10h] \times [0, 10h]$, where $h$ is the space step into the $x$ direction, equal to $h=\frac{5}{2^n}$ where $n$ will be varying in order to built refined meshes.
This choice of length and height of the domain is motivated by the fact that the solution should be independent of those two directions.
It is possible to take a low amount of cells in these directions to reduce the number of unknowns, while still keeping the mesh step in these directions proportional to the mesh step in the $x$ direction.
The mesh is then built as follows:
\begin{itemize}
    \item first, a $10 \times 10$ Cartesian grid of the $y,z$ plan is built, using squares of side length equal to $h$ ;
    \item then, a distortion is applied to this grid ;
    \item an extrusion in $2^n$ cells is then applied to the grid into the $x$ direction, to retrieve the domain $\Omega_{1D}$ ;
    \item finally, the hexahedra obtained in the previous step are refined into either two prisms or two pyramids, depending on the kind of cell types wanted.
\end{itemize}
We also test this problem on a hybrid mesh, composed of hexahedral, pyramidal and prismatic cells. 
To do so, the considered domain is this time $\Omega' = \Omega_{1D} \times [0, 9h] \times [0, 9h]$.
The mesh is constructed in the following way: first, three meshes of height $3h$ are build then glued together along the slices so as to completely mesh the domain $\Omega'$.
These meshes are constructed as previously, and are composed in order of hexahedra, pyramids and prisms.
As a result, the global mesh is deformed, so that the layers of meshes are not stacked in a Cartesian way, but rather by forming layers closer to what can be found in practical cases of use of hybrid meshes.
This transformation sends in particular $\partial \Omega$ over $\partial \Omega$, and is given by 
\begin{align}
   T(x,y,z) =
   \begin{pmatrix}
      1. + 0.2 \sin( \frac{\pi x}{5} ) \sin( \frac{ 2 \pi z}{9 h} ) x \\
      y \\
      z
   \end{pmatrix}
\end{align}

The initial and boundary conditions are then prescribed as previously: the two last components of the velocity are set to zero. 
The analytical solution can be obtained in the same fashion.
On the newly created boundary, wall boundary conditions of type slip are enforced.

MUSCL approximations for both convection operators are chosen.
A stabilization term is added to the discrete momentum balance equation, which is of the form:
\begin{align*}
    \sum_{\substack{\edged \in \edgesd(D_\edge),\\ \edged=D_\edge|D_{\edge'}}} \nu_{\edged}^{n+1}(u_{\edge,i}^n - u_{\edge',i}^n)
\end{align*}
where $\nu_{\edged}$ is an user fixed parameter set to $\nu_{\edged} = \frac{|u_{R1} \rho_{R1}|}{50}$.
The computation is ran until $T=0.015$, which is slightly lower than $T_{\mathrm max}$.

The convergence rate in $L^1(\Omega)$ for the pressure and the density, as well as the convergence rate in discrete $L^1(\Omega)^d$ norm for the velocity are computed at times $t_0=0.003$ (which is before the shock bounces on the left part of the domain) and $t_1=0.015$.
The latter norm is computed by summing on each cell the volume of the cell multiplied by the value of the function at the gravity center of the cell.
Moreover, we use in each case a relative norm, that is we divide the norm of the error by the norm of the exact solution over the whole domain.
Indeed, since the length and height of the domain are scaled with the space step, the error norms (as well as the norm of the exact solution) are scaled with an additional factor proportional to $h^2$.
By taking the relative norms, we ensure that the computed convergence rate is not impacted by this factor.
The sequence of meshes is built by taking $n \in \llbracket 6,11 \rrbracket$.
The error curves for the different variables on prismatic meshes (resp. pyramidal, hybrid) can be found on Fig. \ref{fig:error_choc_prism} (resp. Fig. \ref{fig:error_choc_pyr}, Fig.\ref{fig:error_choc_hyb}.
These results are satisfying, since on both sequence of meshes and for all quantities a convergence rate close to 1 is recovered.

\begin{figure}
    \centering
    \includegraphics[width = 0.3 \textwidth]{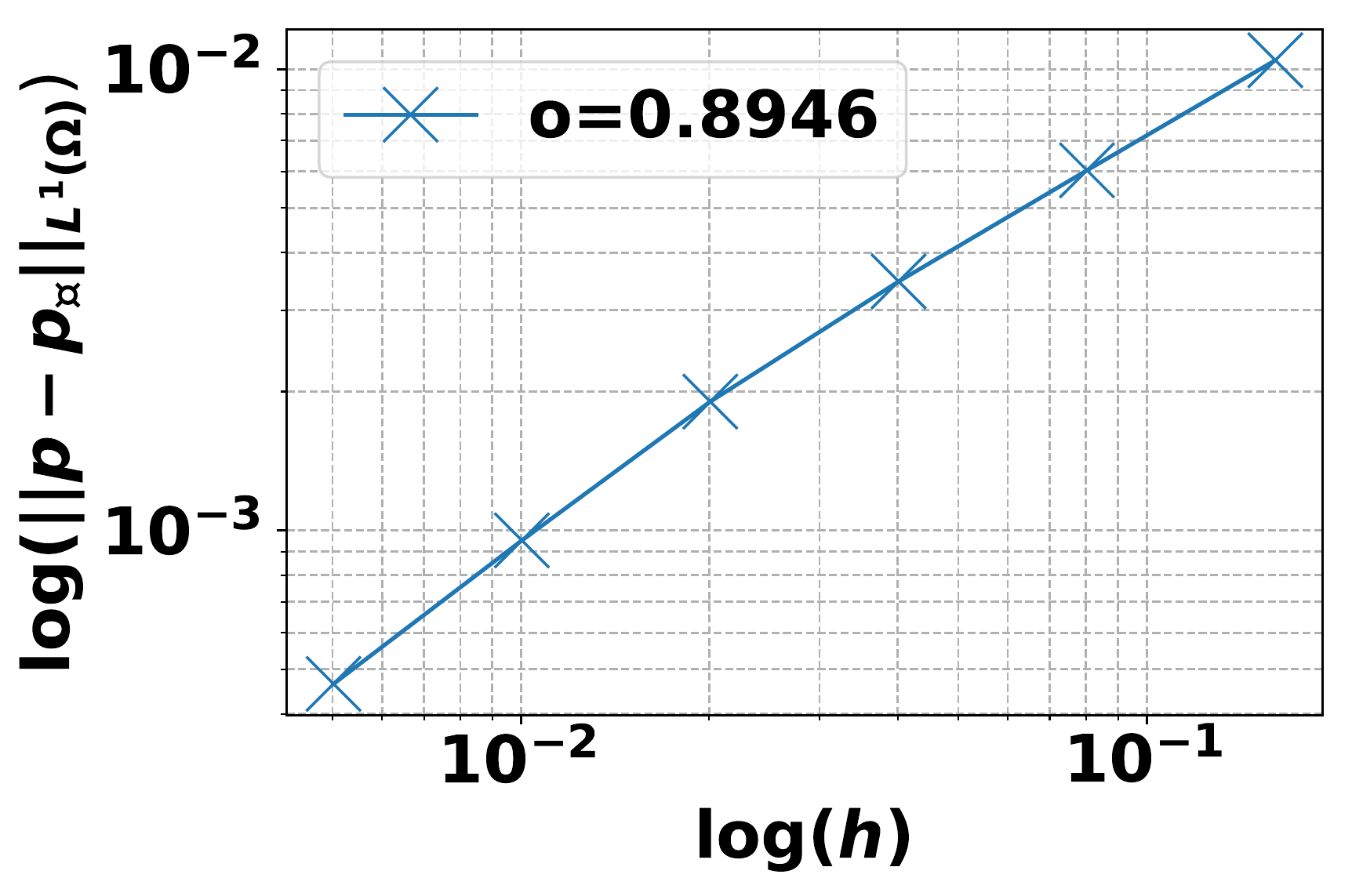}%
    \includegraphics[width = 0.3 \textwidth]{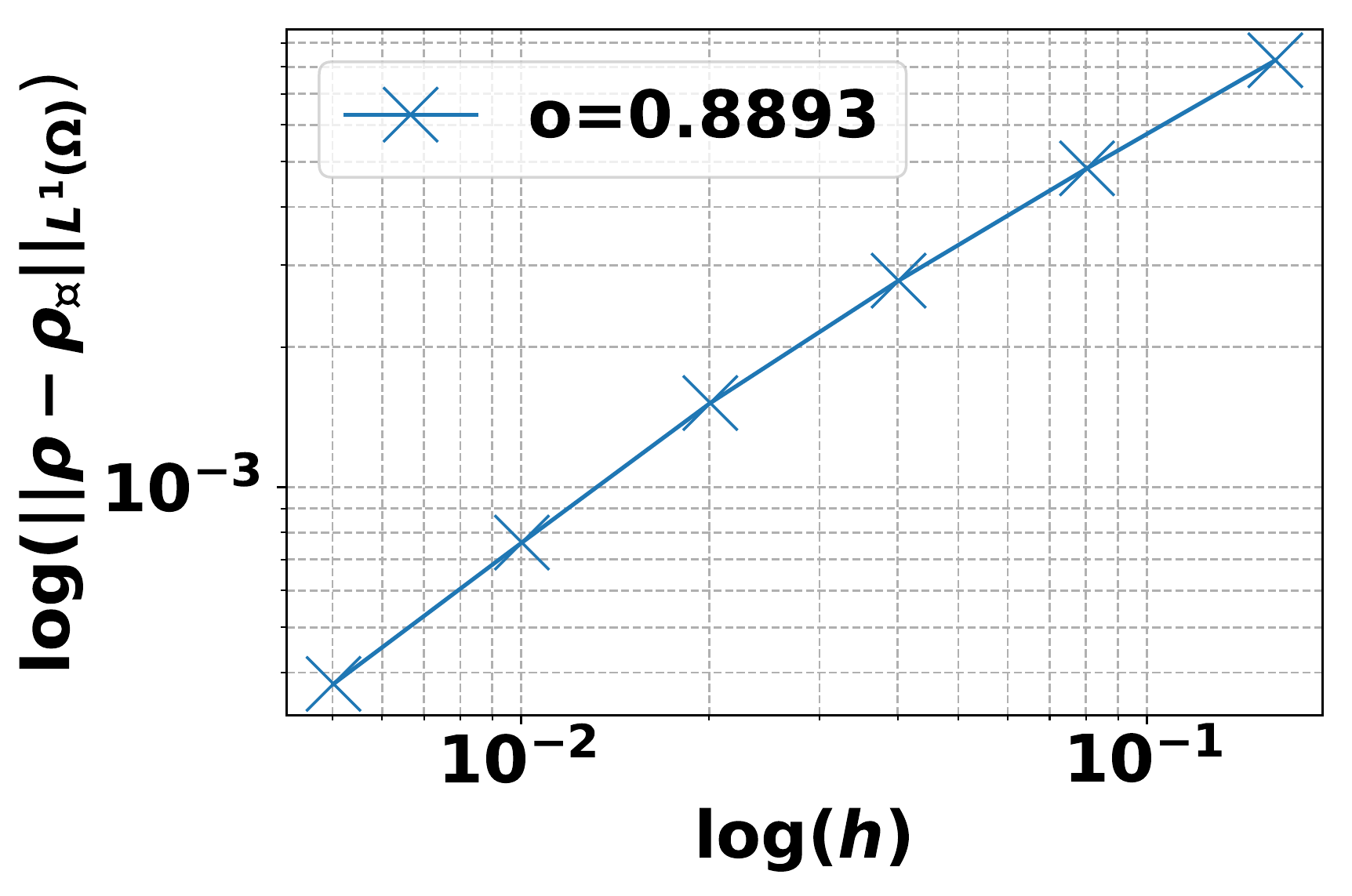}%
    \includegraphics[width = 0.3 \textwidth]{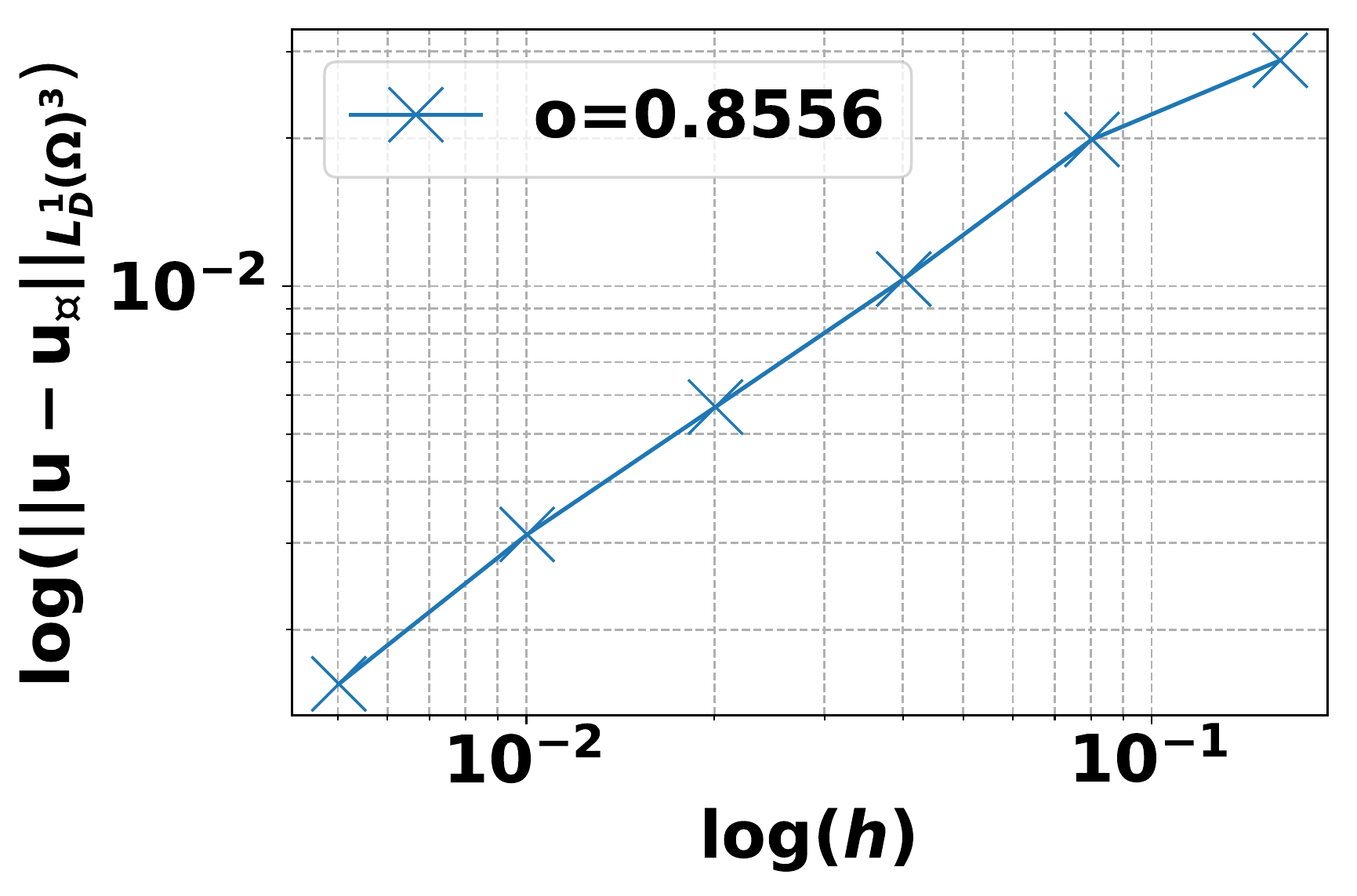}%
    \\
    \includegraphics[width = 0.3 \textwidth]{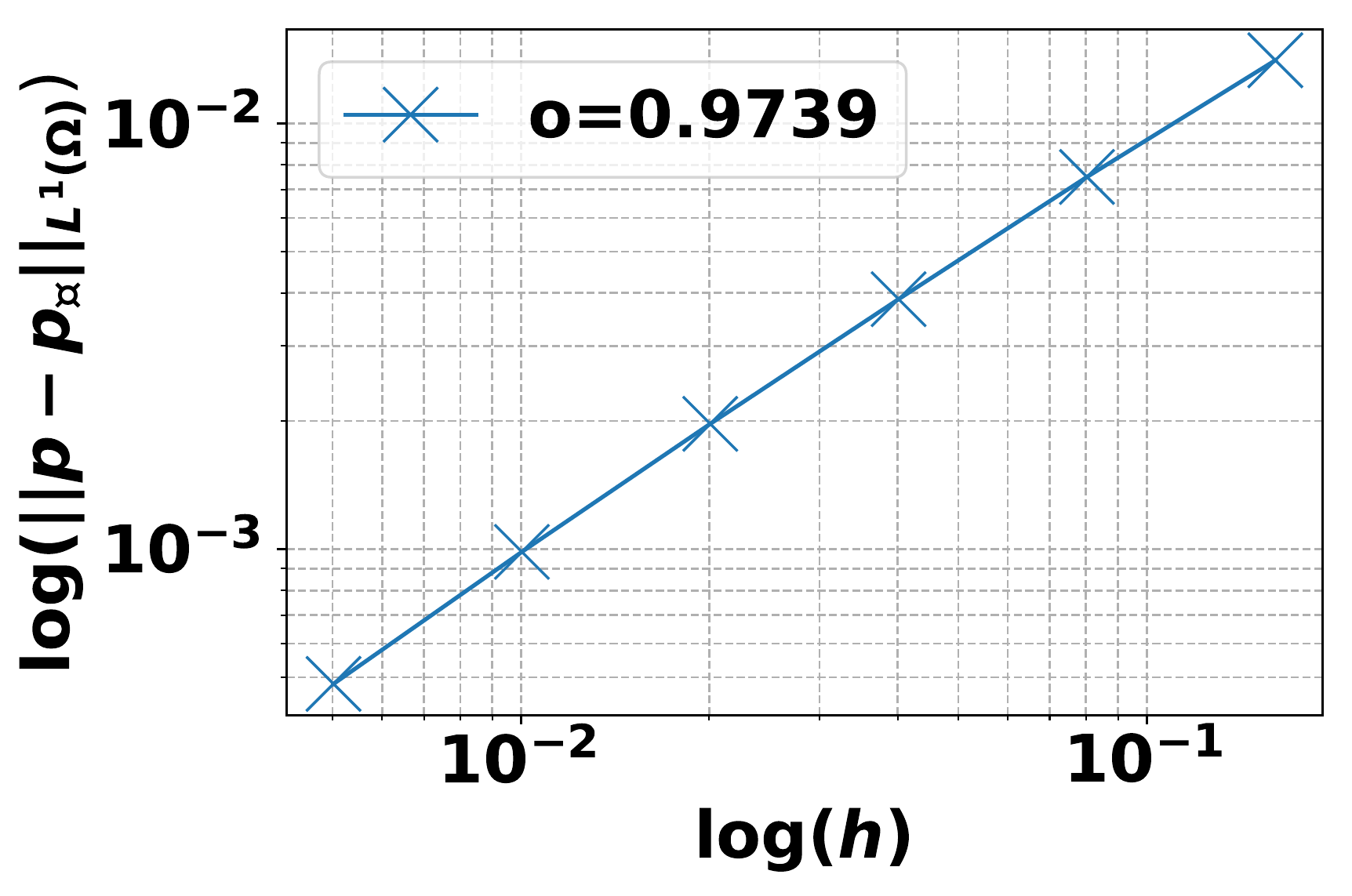}%
    \includegraphics[width = 0.3 \textwidth]{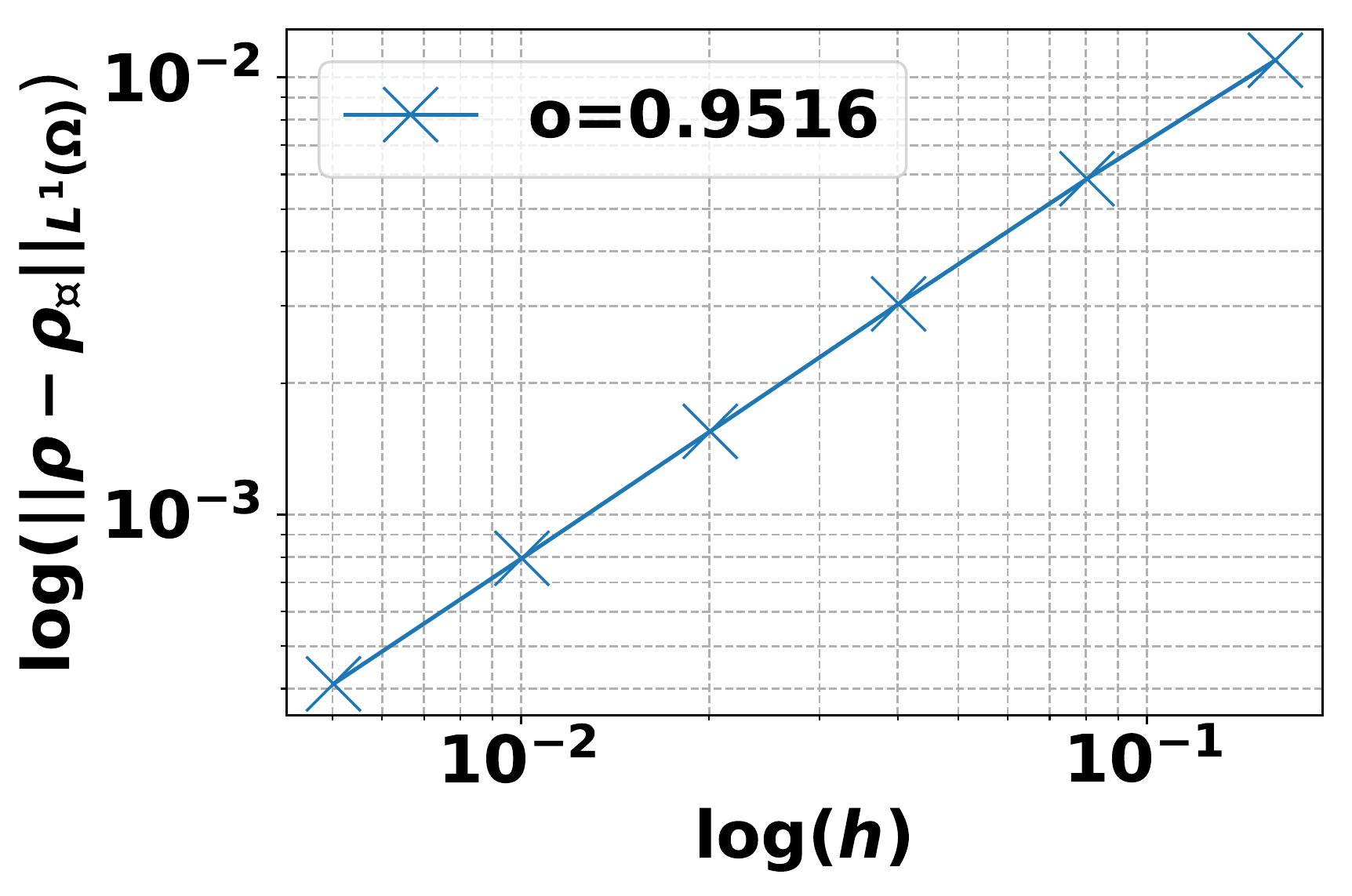}%
    \includegraphics[width = 0.3 \textwidth]{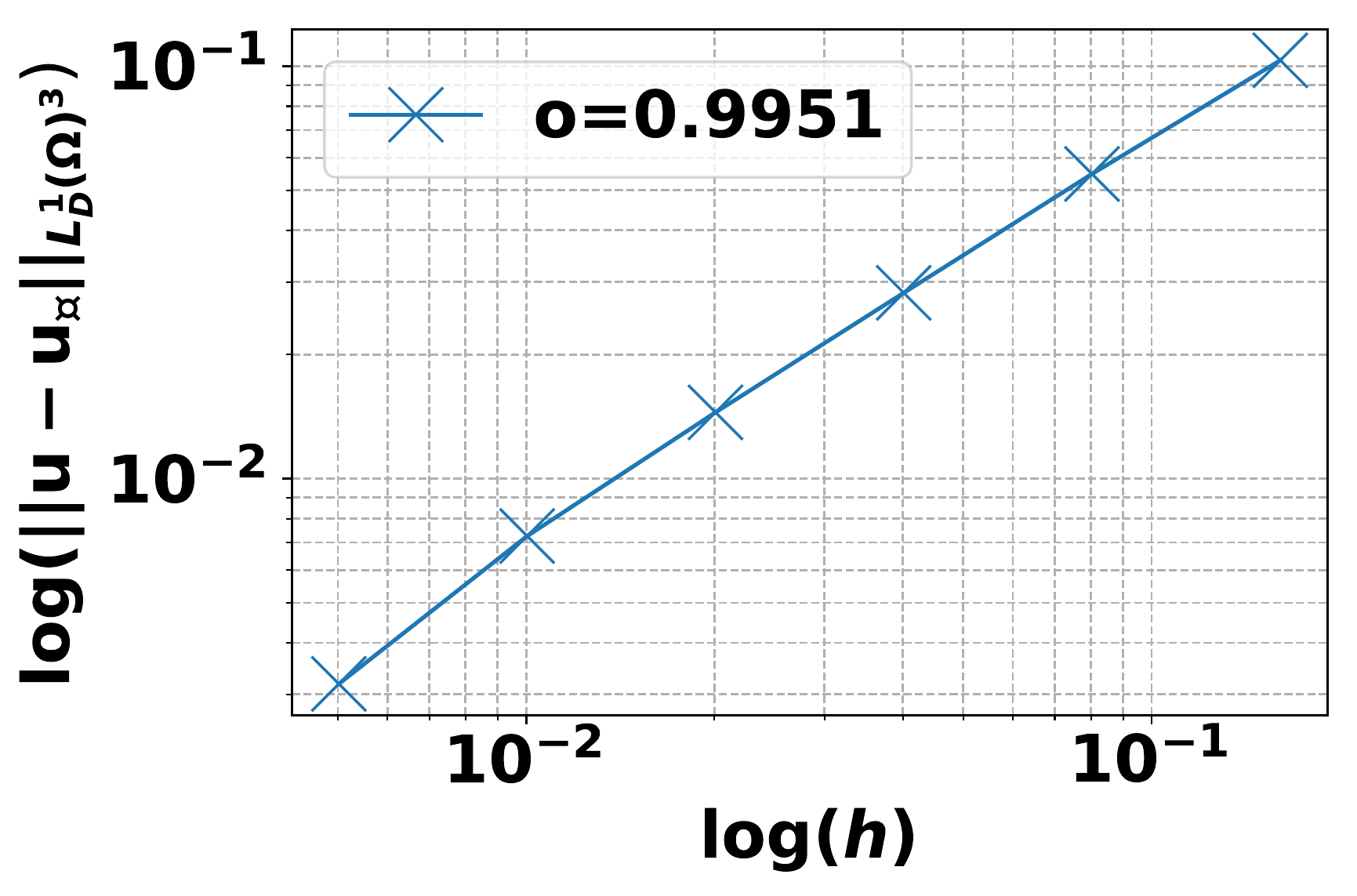}%
    \caption{Convergence rate for the problem of a shock with reflexive boundary condition on a sequence of prismatic meshes. From left to right : pressure in $L^1(\Omega)$ norm, density in $L^1(\Omega)$ norm, velocity in discrete $L^1(\Omega)^d$ norm. From top to bottom : error at time $t=0.003$, errors at time $t=0.015$. }
    \label{fig:error_choc_prism}
\end{figure}

\begin{figure}
    \centering
    \includegraphics[width = 0.3 \textwidth]{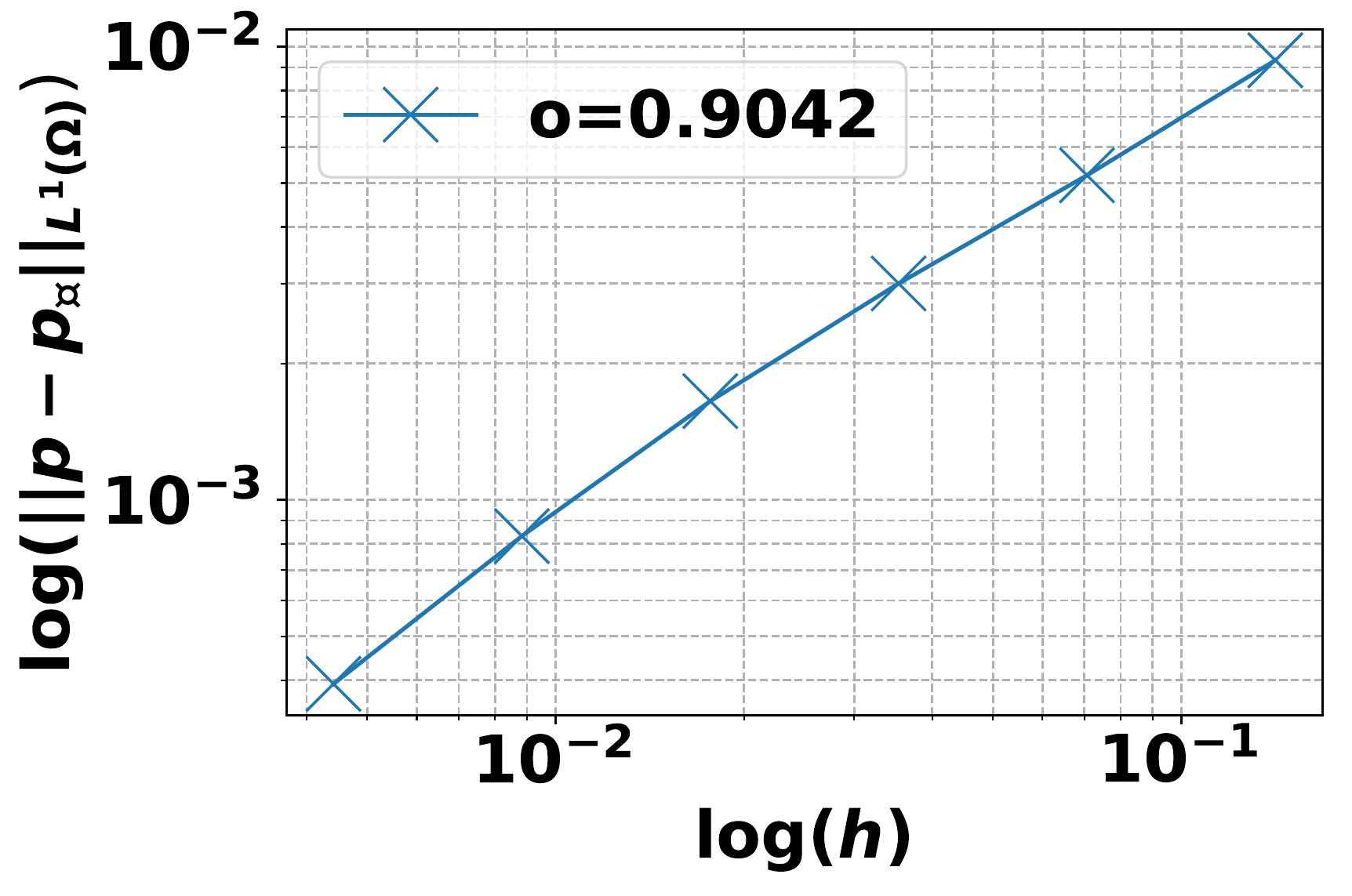}%
    \includegraphics[width = 0.3 \textwidth]{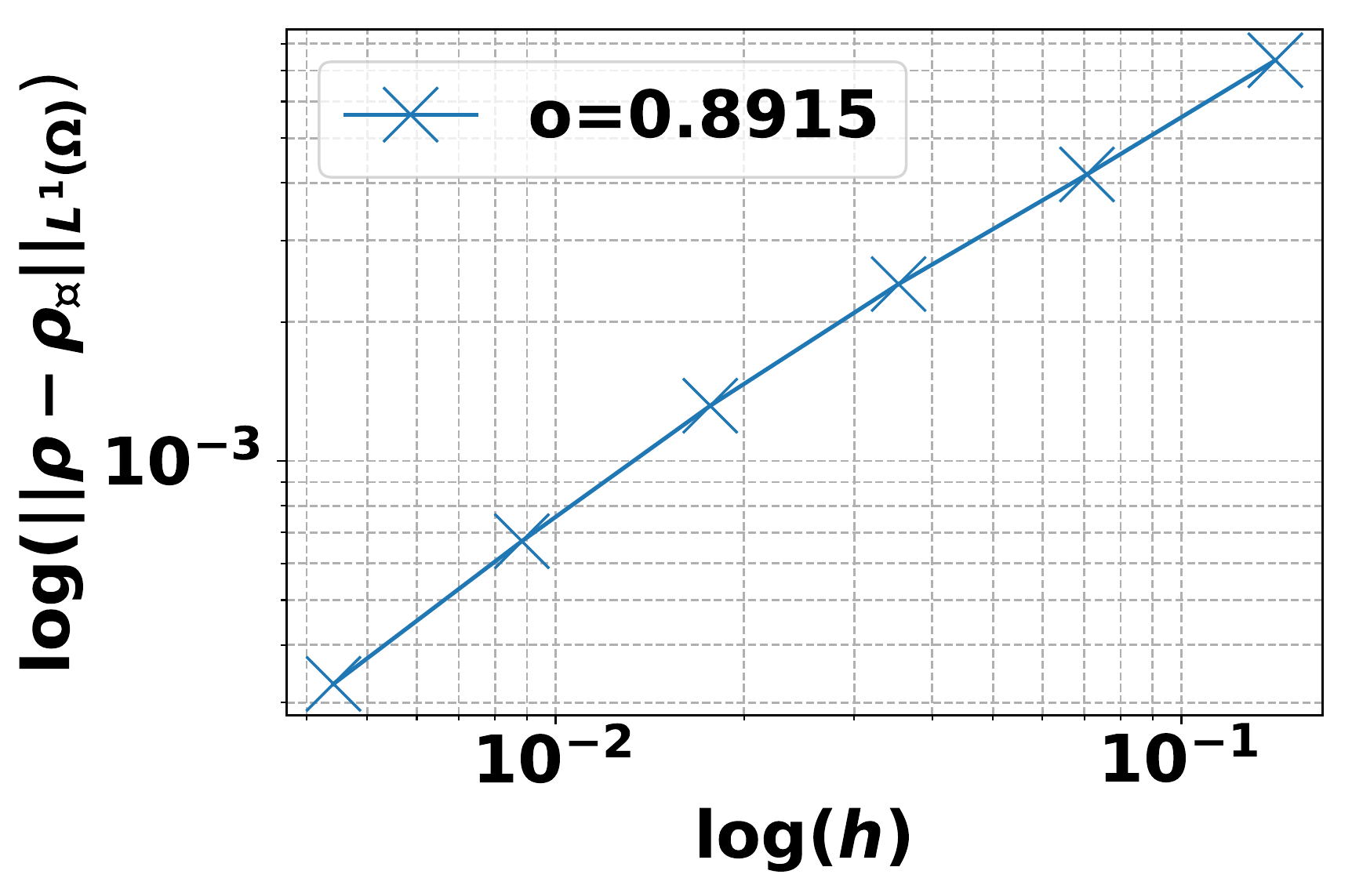}%
    \includegraphics[width = 0.3 \textwidth]{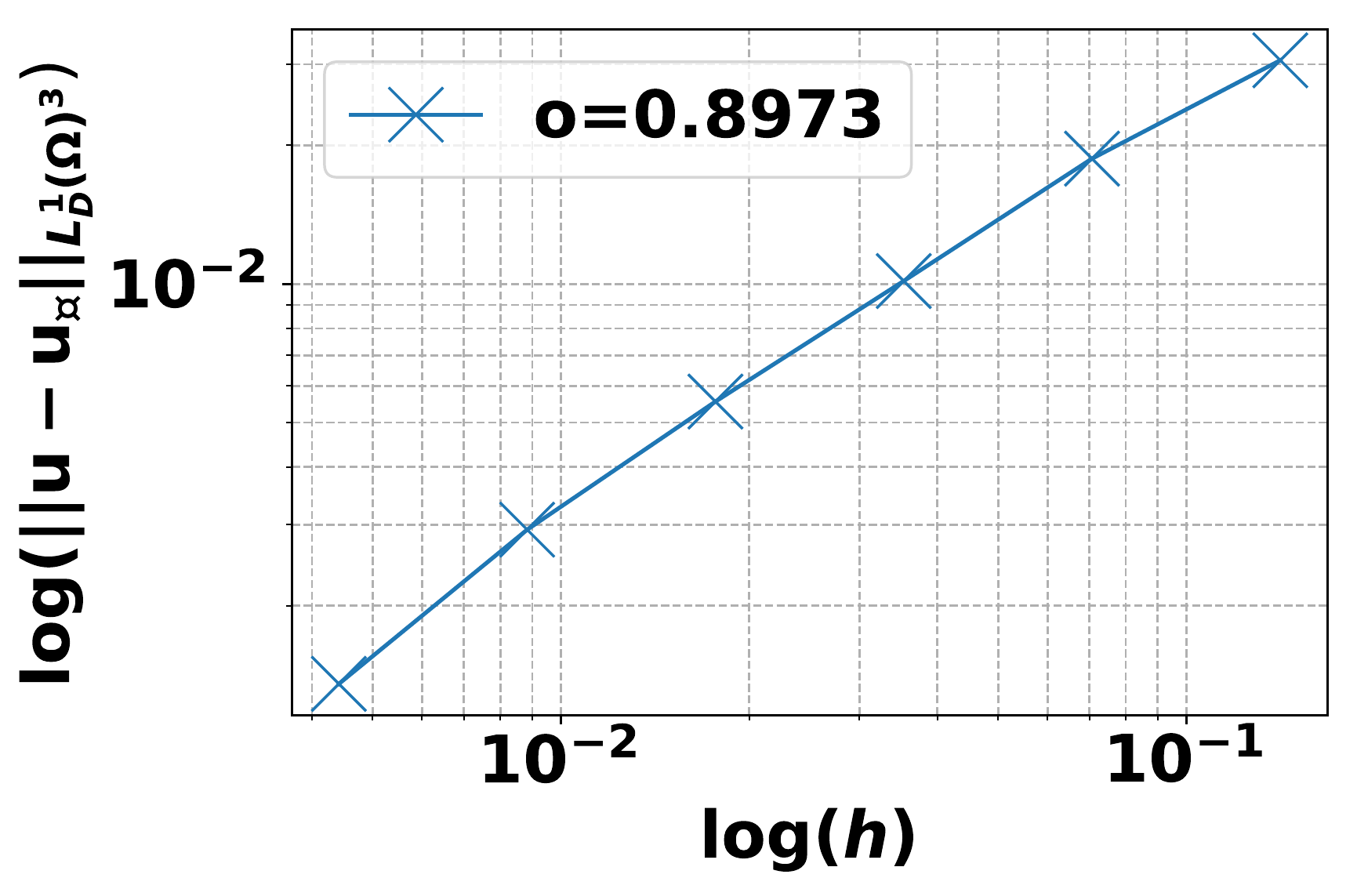}%
    \\
    \includegraphics[width = 0.3 \textwidth]{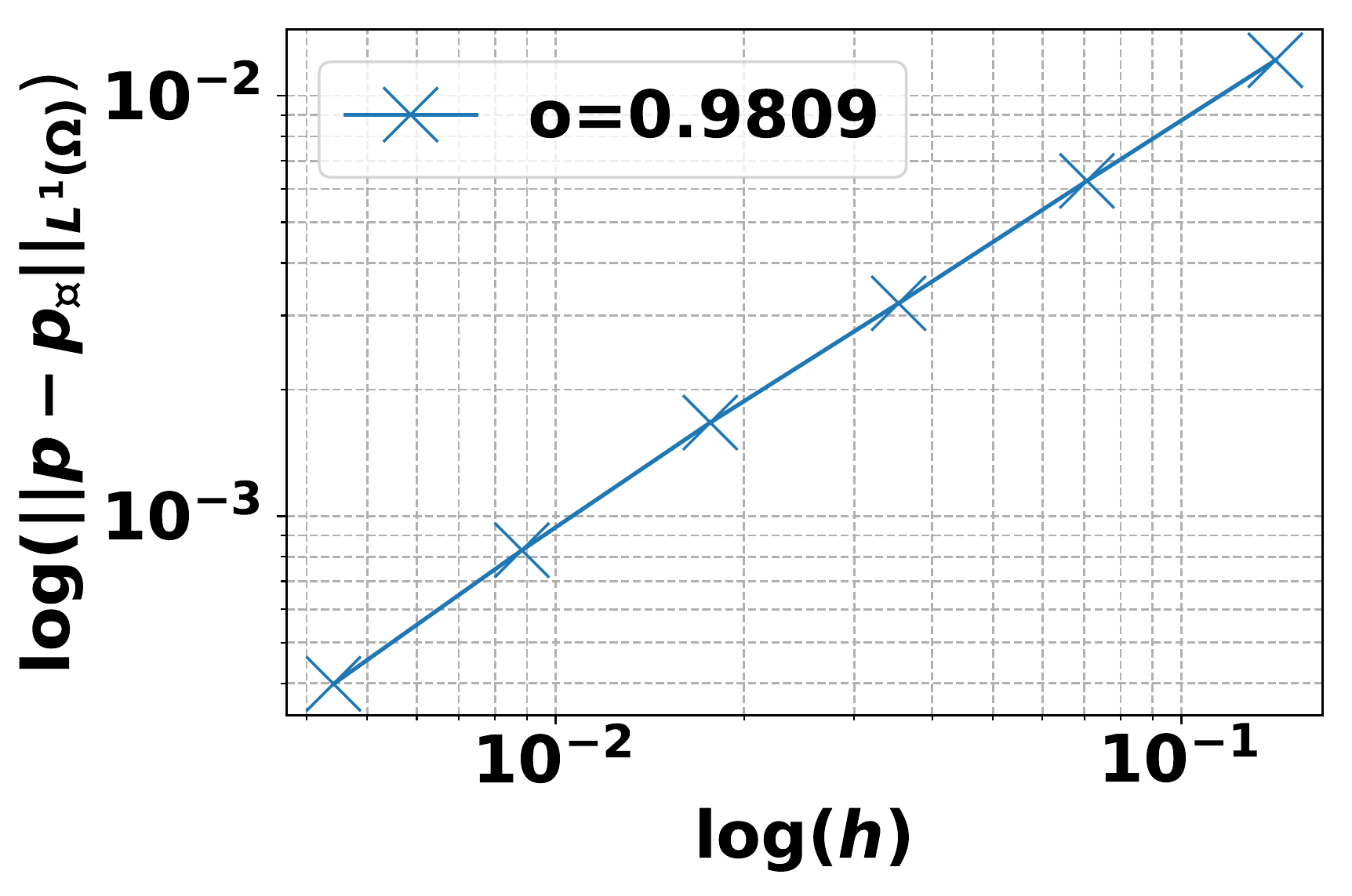}%
    \includegraphics[width = 0.3 \textwidth]{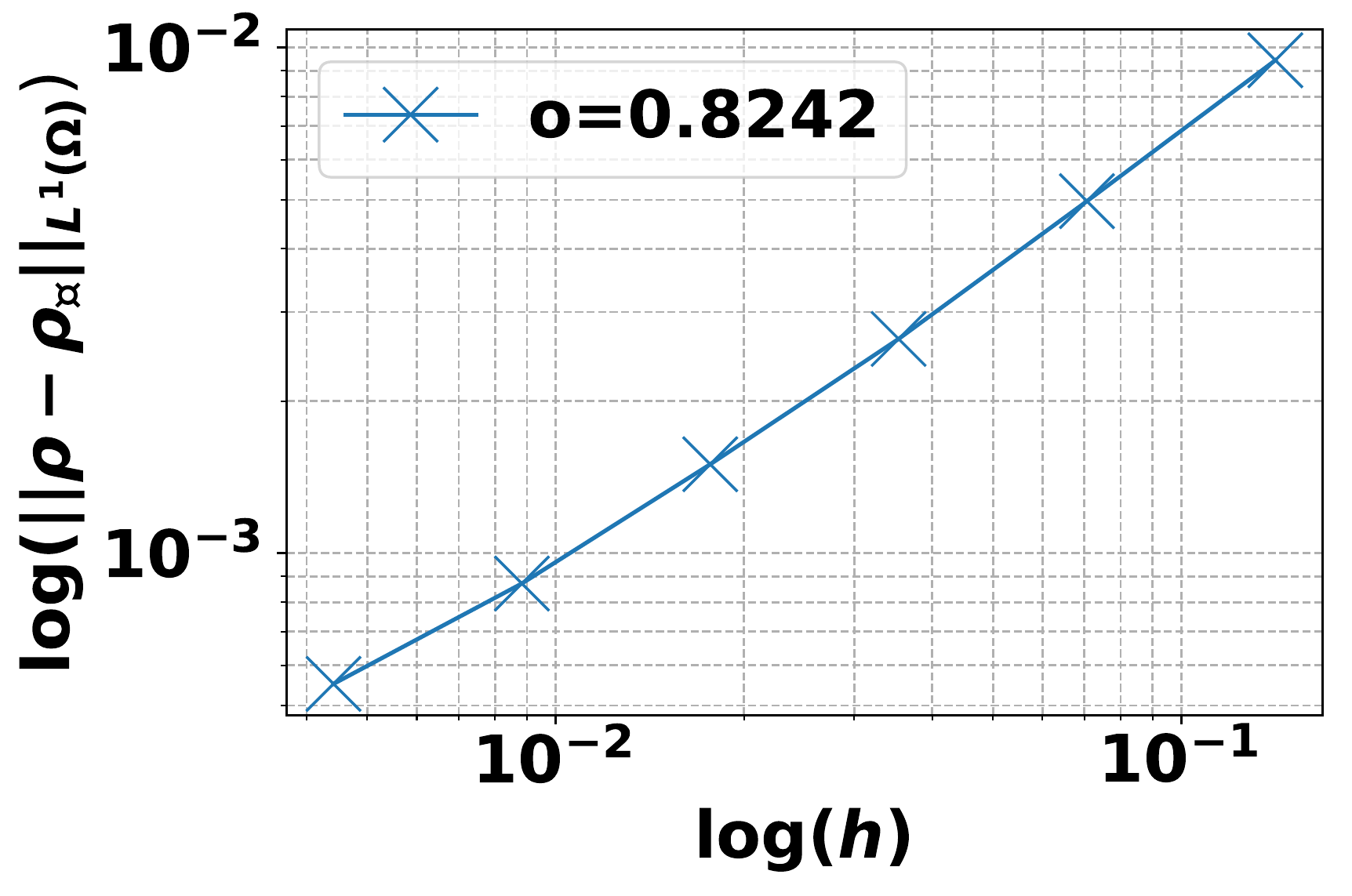}%
    \includegraphics[width = 0.3 \textwidth]{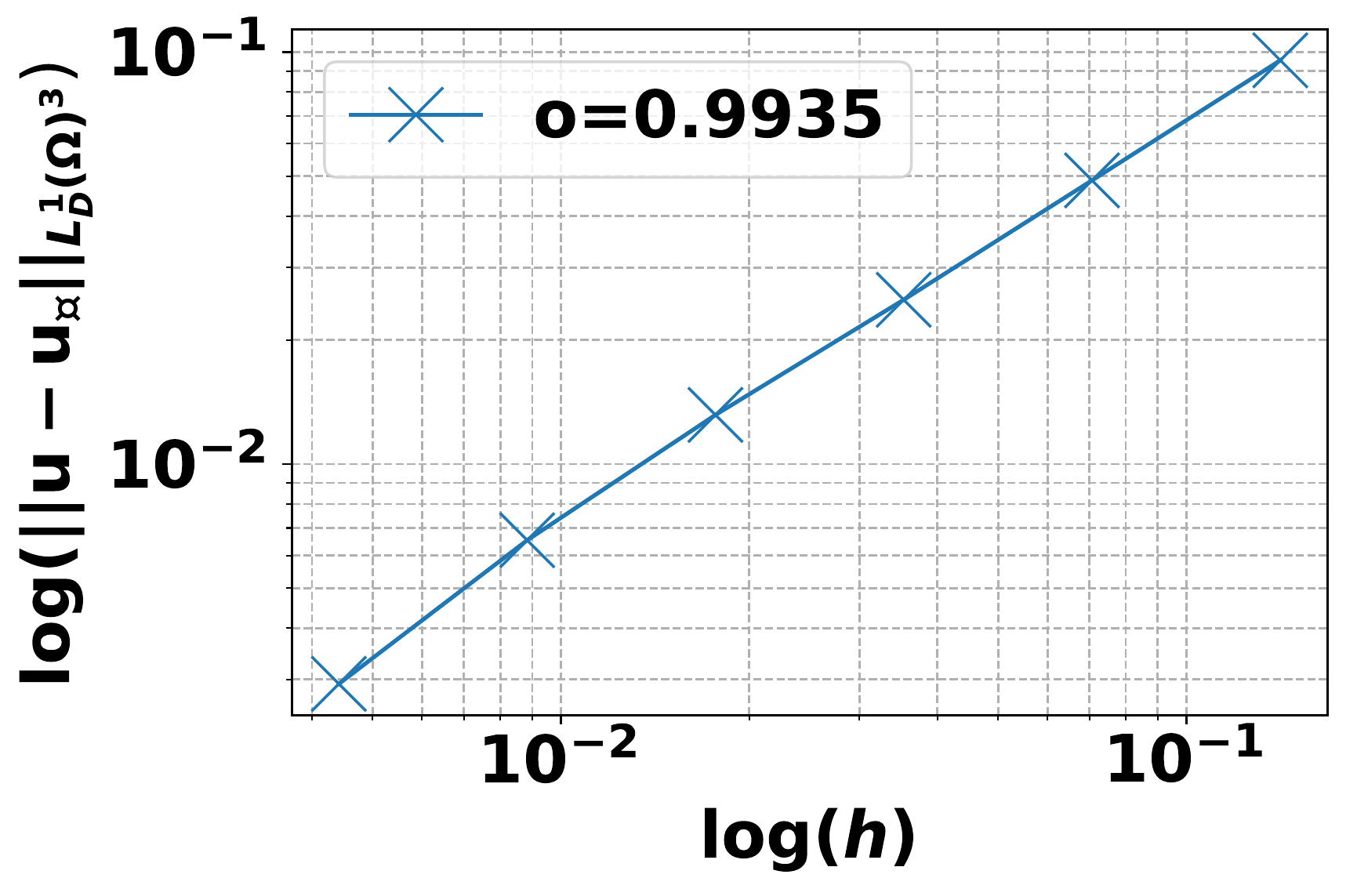}%
    \caption{Convergence rate for the problem of a shock with reflexive boundary condition on a sequence of pyramidal meshes. From left to right : pressure in $L^1(\Omega)$ norm, density in $L^1(\Omega)$ norm, velocity in discrete $L^1(\Omega)^d$ norm. From top to bottom : error at time $t=0.003$, errors at time $t=0.015$. }
    \label{fig:error_choc_pyr}
\end{figure}

\begin{figure}
    \centering
    \includegraphics[width = 0.3 \textwidth]{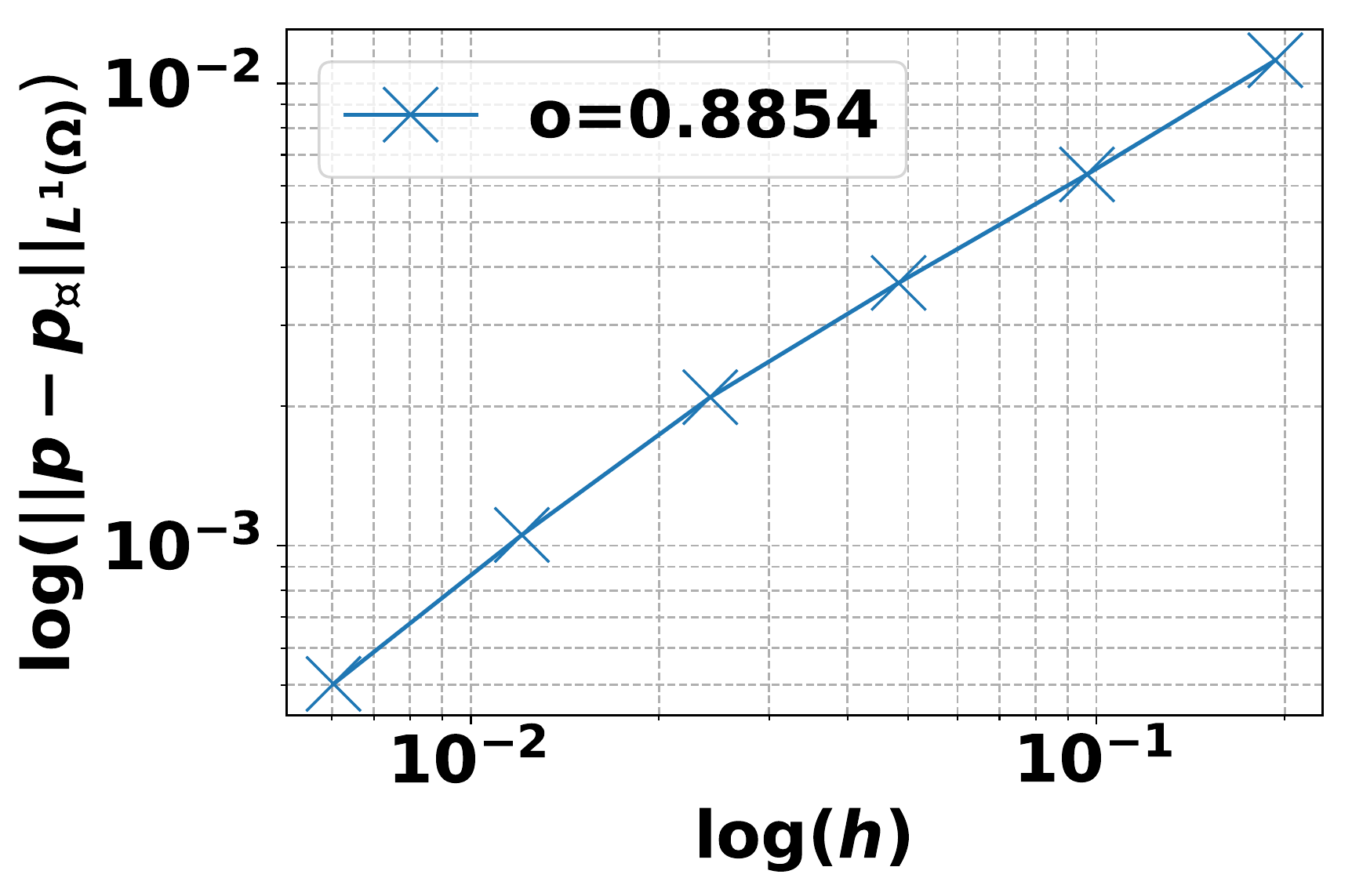}%
    \includegraphics[width = 0.3 \textwidth]{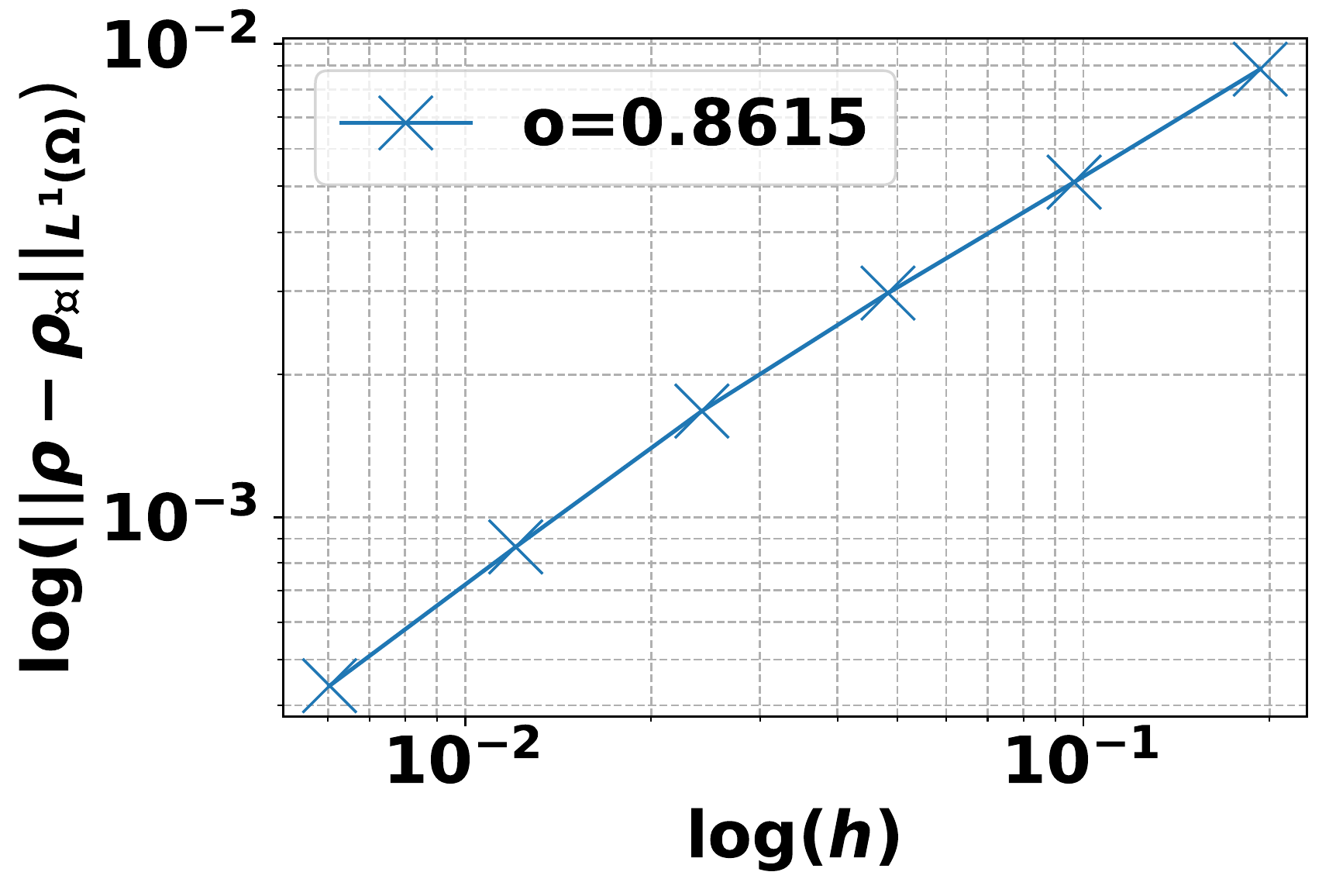}%
    \includegraphics[width = 0.3 \textwidth]{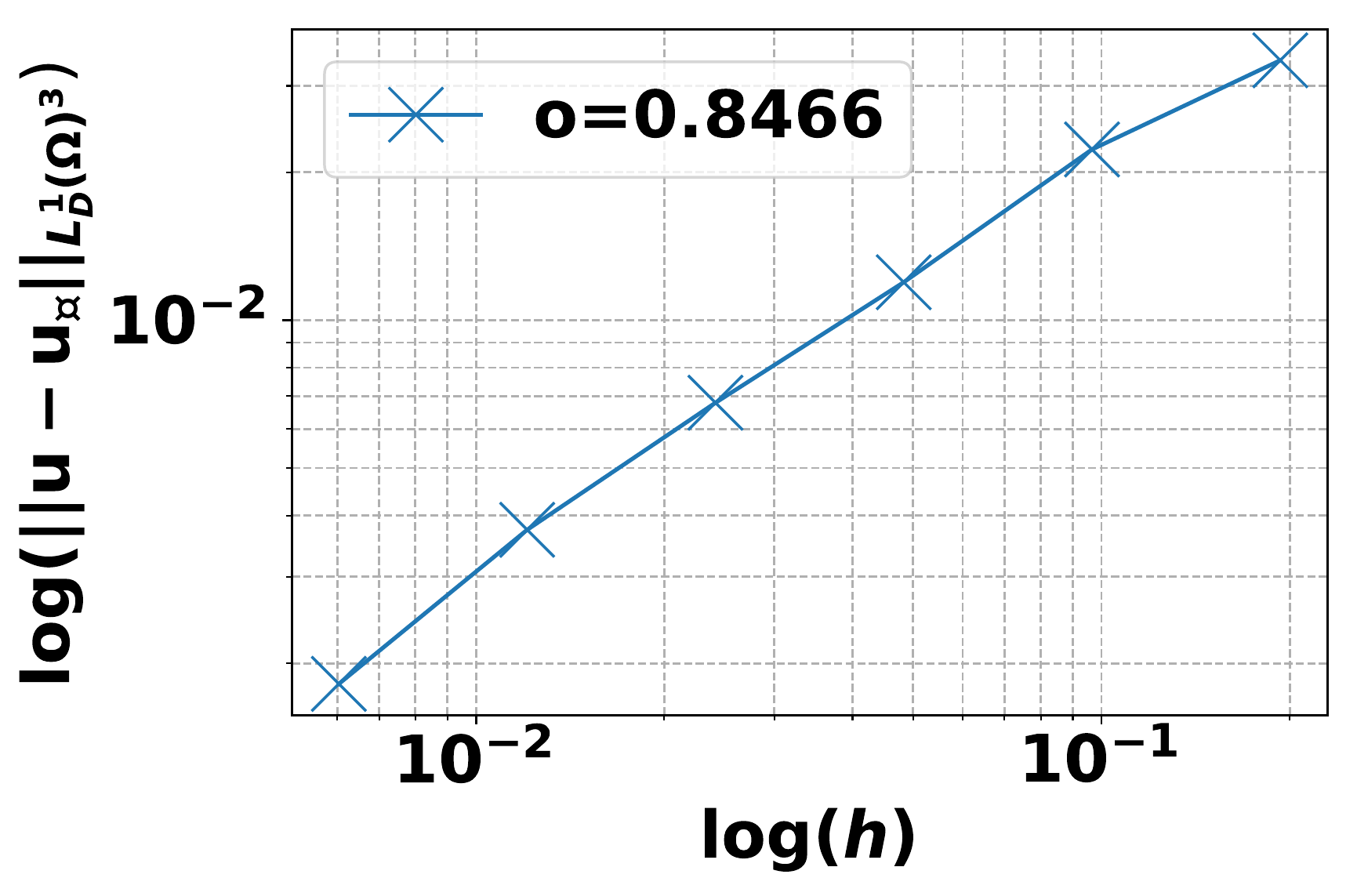}%
    \\
    \includegraphics[width = 0.3 \textwidth]{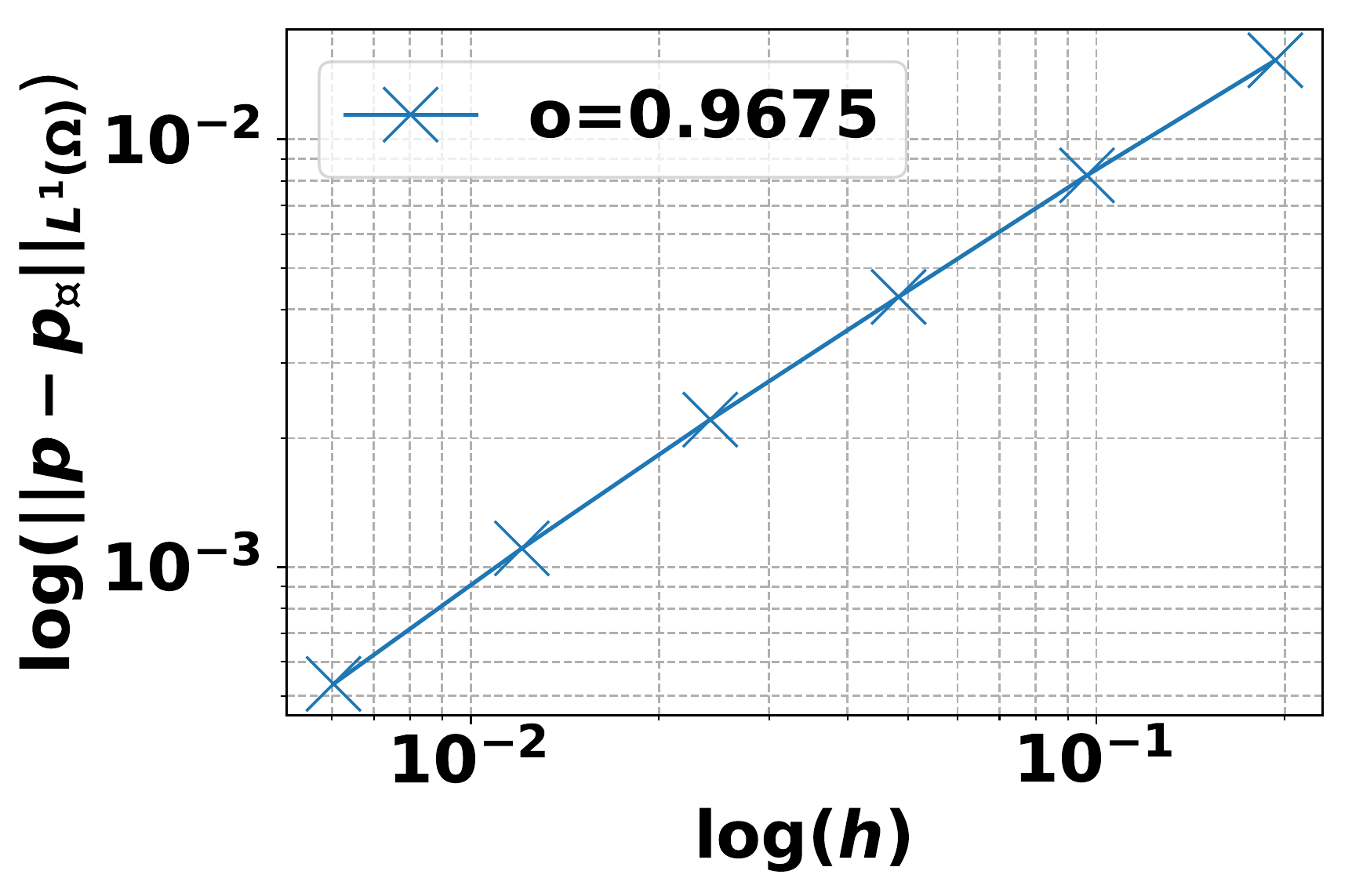}%
    \includegraphics[width = 0.3 \textwidth]{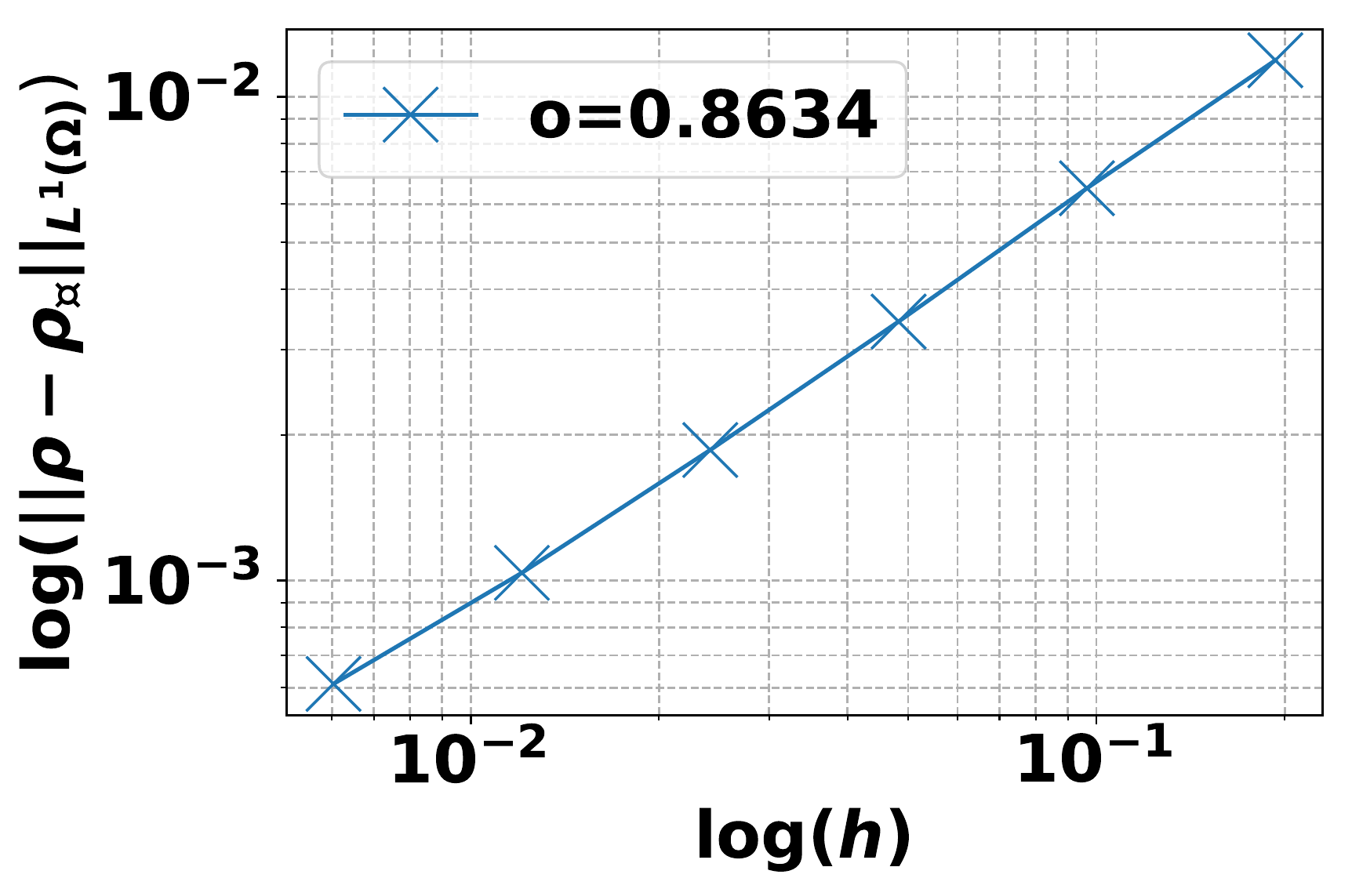}%
    \includegraphics[width = 0.3 \textwidth]{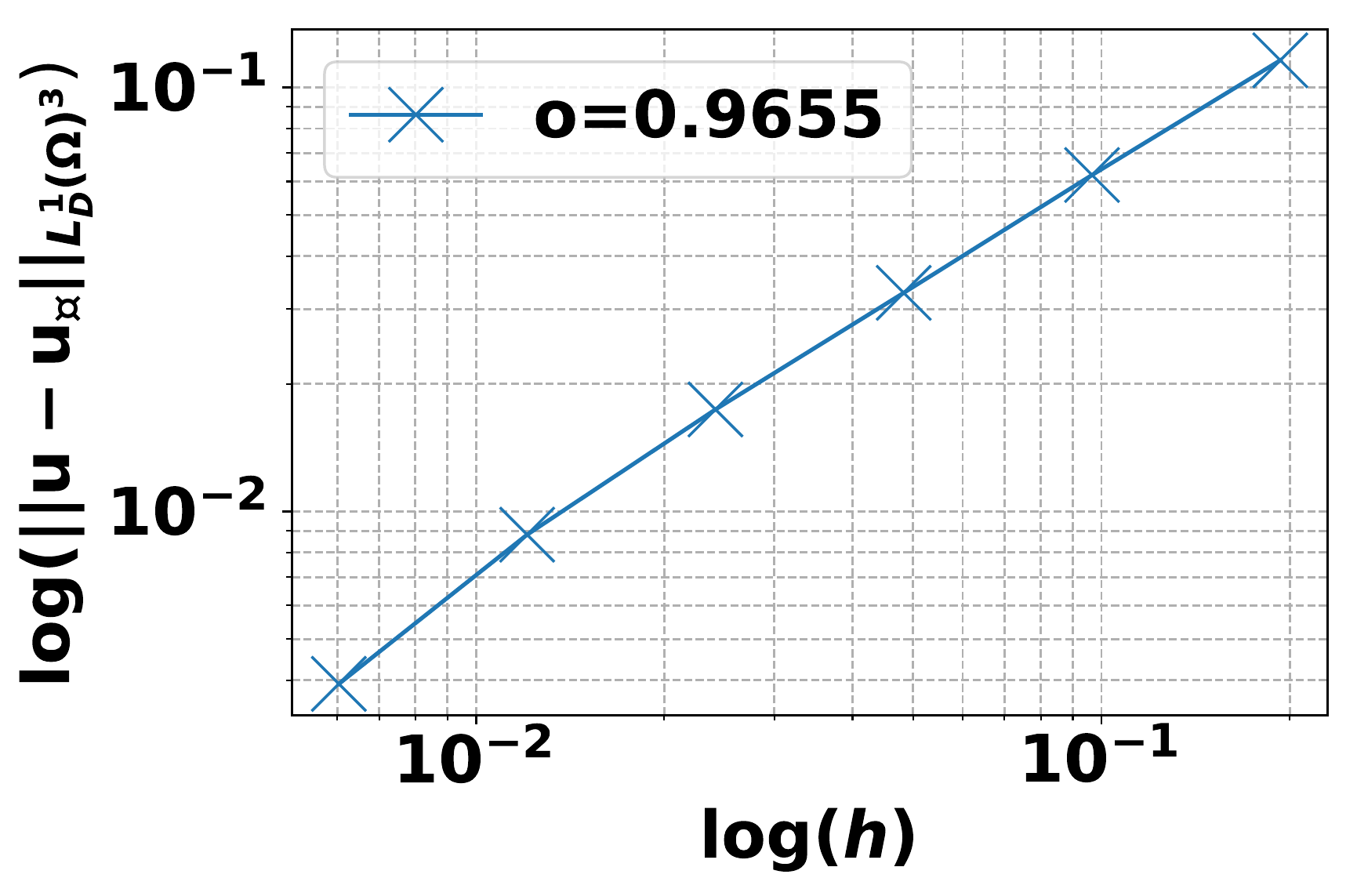}%
    \caption{Convergence rate for the problem of a shock with reflexive boundary condition on a sequence of hybrid meshes. From left to right : pressure in $L^1(\Omega)$ norm, density in $L^1(\Omega)$ norm, velocity in discrete $L^1(\Omega)^d$ norm. From top to bottom : error at time $t=0.003$, errors at time $t=0.015$. }
    \label{fig:error_choc_hyb}
\end{figure}

\appendix
\section{Technical lemmas}
We recall here some results obtained in  \cite{gal-19-wea} and \cite{gal-22-lax} which generalize the Lax-Wendroff consistency to multidimensional colocated or staggered grids.

We rephrase here the consistency results proven in \cite{gal-22-lax}, in a form that is adapted to the present setting. 
We suppose that:
\begin{equation}\label{pb-general}
    \Omega \subset \xR^d, \; d= 1, 2,3, \ T \in (0,+\infty),
    \ p \in \xN^\ast,\ \beta \in C^0(\xR^p, \xR),\ \bff \in C^0(\xR^p, \xR^d),
\end{equation}
and we consider the conservative convection operator $\bar{\mathcal C}(\bar U)$ acting on a vector $\bar U \in \xR^p$ of functions, real-valued, and defined (in the distributional sense), for $\bar U \in L^\infty(\Omega\times(0,T), \xR^p)$, by: 
\begin{align} \label{def:op-conv}
    \bar{\mathcal C} (\bar U): &\quad \Omega\times(0,T) 	\to \xR, \nonumber
    \\ & \quad 
    (\bfx,t) \mapsto \partial_t\bigl(\beta(\bar U(\bfx,t))\bigr) + \dive\bigl(\bff(\bar U(\bfx,t))\bigr).
\end{align}

Let us denote by $(\mathcal P\exm)\mnn$ a sequence of meshes of the domain $\Omega$, each mesh consisting of a set of disjoint open polyhedral or polygonal open subsets of $\Omega$, whose union of closures is $\bar \Omega$.
We denote by $\delta(\mathcal P)$ the space step, defined by
\[
\delta(\mathcal P\exm) = \max_{P \in \mathcal P\exm}\ \diam (P).
\] 
Let $\edgespart\exm$ denote the set of faces (in 3D, or edges in 2D) of the mesh, and $\edgespartint\exm$ denote the set of faces that are not located on the boundary $\partial\Omega$; for a given polyhedron (or polygon) $P \in \mathcal P\exm$, also called a cell, let $\edgespart(P)$ denote its set of faces (or edges).
Let  $\deltat\exm = \frac{T}{N_m} $, with $N_m \in \xN$, $\deltat\exm \to 0$ as $m \to +\infty$, and let $t_n = n \deltat\exm$ for $n \in \llbracket 0, N_m \rrbracket$. 
 
The unknown is supposed to be represented by a function $U \in L^\infty(\Omega\times(0,T), \xR^p)$ (we take $U= (\rho,\bfu)$ in Section \ref{sec:consis}). 
Note that the unknowns do not need to be piecewise-constant over the cells of the mesh and over the time steps.

\medskip 

The first result defines the weak consistency property, also referred to as consistency in the Lax-Wendroff sense, and provides a set of assumptions which are sufficient for a discrete convection operator (to be understood as "the discrete counterpart of a first order conservative differential operator") to enjoy this property.
The second result states a convergence result which is useful to check these assumptions, namely the convergence to zero of "discrete time and space tranlates" (see below for a definition) of a converging sequence of functions in $L^1$.
We write these theorem with specific notations for the space and time discretizations.
%
%

\subsection{Consistency results}

The discrete convection operator that we consider here takes the following form: 
\begin{align*}
    \mathcal C(U) : 
    & \quad
    \Omega\times(0,T) 	\to \xR,
    \\ & \quad 
    (\bfx,t) \mapsto
    \mathcal C(U)_P^n, \qquad \mbox{ for } \bfx \in P,\ P \in \mathcal P, \mbox{ and } t \in (t_n,t_{n+1}),\ n \in \llbracket 0, N-1 \rrbracket,
\end{align*}
with
\[
    \mathcal C(U)_P^n = \frac{\beta_P^{n+1} - \beta_P^n}{\deltat} + \frac 1 {|P|} \sum_{\zeta \in \edgespart(P)} |\zeta|\ \bfF_\zeta^n \cdot \bfn_{P,\zeta},
\]
where $\bigl\{\beta_P^n,\ P \in \mathcal P,\ n \in \llbracket 0, N \rrbracket \bigr\}$ is a family of real numbers, $\bigl\{\bfF_\zeta^n,\ \zeta \in \edgespart,\ n \in \llbracket 0, N-1 \rrbracket \bigr\}$ is a family of real vectors of $\xR^d$ and $\bfn_{P,\zeta}$ stands for the normal vector to $\zeta$ pointing outward $P$.
Note that this form of the flux implies that the scheme is conservative.
Of course, the real numbers $\bigl\{\beta_P^n,\ P \in \mathcal P,\ n \in \llbracket 0, N \rrbracket \bigr\}$ and $\bigl\{\bfF_\zeta^n,\ \zeta \in \edgespart,\ n \in \llbracket 0, N-1 \rrbracket \bigr\}$ are related to the unknown $U$.

The consistency of this discrete convection operator is proven in 
\cite[Theorem 2.1]{gal-22-lax}, which also states the assumptions that must be satisfied by the above defined quantities to ensure the consistency of the discrete convection operator  $\mathcal C(U)$. 
Unfortunately we cannot apply it directly in the framework of our paper, since the third assumptions of this theorem involves a normal flux on the faces of the dual mesh, which are in general not defined in the schemes we are studying. 
To overcome this difficulty, we use two intermediate results of \cite[Theorem 2.1]{gal-22-lax}. 
The first one, \cite[Lemma 2.7]{gal-22-lax}, deals with the consistency of the discrete time derivative; it is used in Section \ref{sec:consis} with the dual mesh to show the consistency of the time derivative of the momentum. 
The second one, \cite[Lemma 2.8]{gal-22-lax} deals with the consistency of the nonlinear divergence term in the momentum equation, and is used in Section \ref{sec:consis} with the primal mesh because of the non existence of normal vectors for general dual meshes. 
These two lemmas require the following common assumptions: 
for a sequence $(\mathcal P\exm,\mathcal T\exm)_\mnn$ of space-time discretisations, with $\delta(\mathcal P\exm), \deltat\exm \to 0$ as $m\to +\infty$,  let $(U\exm)_\mnn$ be the associated sequence of discrete functions.
We suppose that the sequence $(U\exm)_\mnn$ is bounded and converges to a limit:
there exists $\ctel{cons-u} \in \xR_+^\ast$  such that 
\begin{subequations}
\label{hyp-lemmas}
 \begin{align} 
\label{lemgen:linfbound} &
     \Vert U\exm \Vert_{L^\infty(\Omega\times(0,T), \xR^p)} \le \cter{cons-u},\ \forall \mnn,
    \\ 
\label{lemgen:l1conv} &
    \exists \ \bar U \in L^\infty(\Omega\times(0,T), \xR^p) \mbox{ s.t. } \Vert U\exm - \bar U \Vert_{L^1(\Omega\times(0,T), \xR^p)} \to 0 \mbox{ as } m \to +\infty.
\end{align}
\end{subequations}

\begin{lemma}[LW-consistency, time derivative] \label{lem:time-cons}
Assuming \eqref{pb-general}, let $(\mathcal P\exm,\mathcal T\exm)_\mnn$ be a sequence of space-time discretisations, with $\delta(\mathcal P\exm)$ and $\deltat\exm$ tending to zero as $m\to +\infty$, and let $(U\exm)_\mnn$ be the associated sequence of discrete functions satisying \eqref{hyp-lemmas}.
Let $U_0 \in L^\infty(\Omega,\xR^p)$  and assume that
  \begin{align} \label{hyp:condi} &
 \sum_{P \in \mathcal P_{\mathrm{int}}\exm}   \int_P |(\beta\exm)_P^0  - \beta(U_0(\bfx)) | \dx  \to 0 \mbox{ as } m \to + \infty,
    \\ \label{hyp:t} &
 \sum_{n=0}^{N_m-1}  \sum_{P \in \mathcal P_{\mathrm{int}}\exm} \int_{t_n}^{t_{n+1}} \int_P |(\beta\exm)_P^n  - \beta(U\exm(\bfx,t)) | \dx \dt \to 0 \mbox{ as } m \to + \infty,
 \end{align}
where $\mathcal P_{\mathrm{int}}\exm$ denotes the set of cells of $ \mathcal P\exm$ that have no face or edge on the boundary $\partial \Omega$.
Then
\begin{multline*}
    \sum_{n=0}^{N\exm-1}  \sum_{P \in \mathcal P\exm }|P|  \left((\beta\exm)_P^{n+1}-(\beta\exm)_P^n\right)\ \varphi_P^n \to - \int_\Omega \beta(U_0)(\bfx)\ \varphi (\bfx,0) \dx
    \\
    - \int_0^T \int_\Omega \beta(\bar U)(\bfx,t)\ \partial_t \varphi(\bfx,t) \dx \dt \quad \mbox{as } m\to +\infty,
\end{multline*}

with
\begin{equation}\label{eqdef:varphiPn}
  \varphi_P^n = \frac{1}{|P|}\int_P \varphi(\bfx,t_n) \dx.
\end{equation}
\end{lemma}

\begin{lemma}[LW-consistency, space derivative] \label{lem:space-cons}
Assuming \eqref{pb-general}, let $(\mathcal P\exm,\mathcal T\exm)_\mnn$ be a sequence of space-time discretisations, with $\delta(\mathcal P\exm)$ and $\deltat\exm$ tending to zero as $m\to +\infty$, and let $(U\exm)_\mnn$ be the associated sequence of discrete functions safisying \eqref{hyp-lemmas}.
Assume furthermore that \begin{equation}   \label{hyp:x} 
    \sum_{n=0}^{N\exm-1} \sum_{P \in \mathcal P_{\mathrm{int}}\exm} \frac {\diam(P)}{|P|}
    \sum_{\zeta \in \edgespart(P)} |\zeta| \int_{t_n}^{t_{n+1}}  \int_P \Bigl| \Big((\bfF\exm)_{\zeta}^n -\bff\bigl(U^m(\bfx,t)\bigr) \Big) \cdot \bfn_{P,\zeta} \Bigr| \dx \dt \to 0, 
\end{equation}
Then 
\begin{align*}
    \sum_{n=0}^{N\exm-1} \deltat\exm \sum_{P \in \mathcal P\exm }\ \sum_{\zeta \in \edgespart(P)} |\zeta|\ (\bfF\exm)_\zeta^n \cdot \bfn_{P,\zeta}\ \varphi_P^n
    \to - \int_0^T \int_\Omega \bff(\bar U)(\bfx,t) \cdot \gradi \varphi(\bfx,t) \dx \dt \quad \mbox{ as } m\to +\infty,
\end{align*}
with
$\varphi_P^n$ defined by \eqref{eqdef:varphiPn}. 
\end{lemma}

If both the assumptions of Lemma \ref{lem:time-cons} and Lemma \ref{lem:space-cons} are satisfied, then the convergence of the weak form of the whole discrete convection operator is ensured, see \cite[Theorem 2.1]{gal-22-lax}. 
However, in our setting, the assumption \eqref{hyp:x} of Lemma \ref{lem:space-cons} cannot be satisfied on  virtual dual meshes, so that Lemma \ref{lem:time-cons} is used on the dual mesh (it does not necessitate the knowledge of the boundaries of the dual cells) while Lemma \ref{lem:space-cons} is used on the primal mesh, for which the assumption \ref{hyp:x} can be checked.
%
\subsection{A bound of discrete translates of discrete functions}

The following result is a consequence of \cite[Lemma A.1]{gal-22-lax}.
It features a mesh $\mathcal P$ which can be either the primal mesh, the dual mesh, or a mesh constructed from the edges of the dual mesh.
For $u \in L^1(\Omega \times (0,T))$, $P \in \mathcal P$ and $n$ such that $n \in \llbracket 0, N-1 \rrbracket$, let $u_P^n$ be the mean value of $u$ over $P \times (t_n,t_{n+1})$.
Let $T_{\mathcal P}\, u$ be defined by
\begin{equation}\label{transT-u_gene}
    T_{\mathcal P}\, u= \sum_{n=0}^{N-1} \deltat \sum_{\edge = P|Q \in \mathcal \edges(\mathcal P)} \omega_{P,Q}\ |u^{n+1}_Q - u^{n+1}_P|
\end{equation}
where $(\omega_{P,Q})_{\edge = P|Q \in \mathcal \edges(\mathcal P)}$ is a set of non-negative weights.
We introduce the two following parameter:
\begin{equation} \label{eq:reg_gene}
    \begin{array}{l} \displaystyle
        \theta_{\mathcal P} = \max_{P \in \mathcal P} \frac 1 {|P|}\ \sum_{\substack{Q \in \mathcal P\\ \{P,Q\} \in \mathcal S_x}} \omega_{P,Q}.
    \end{array}
\end{equation}
 
 \begin{theorem}\label{thm:trans_bound}
    Let $(\mathcal P\exm)_{m \in \xN}$  be a given sequence of meshes with $h\exm = \max_{K \in \mathcal P\exm} \to 0$ as $m \to +\infty$, and let $\deltat\exm \to 0 $as $m \to +\infty$.
    Let us suppose that there exists $\theta > 0$ such that $\theta_{\mathcal P\exm} \le \theta$ for all $m \in \xN$, with $\theta_{\mathcal P\exm}$  by Equation \eqref{eq:reg_gene}.
     
    Let $u \in L^1(\Omega\times(0,T))$ and $(u_p)_{p \in \xN}$ be a sequence of functions of $L^1(\Omega\times(0,T))$ such that $u_p \to u$ in $L^1(\Omega\times(0,T))$ as $p \to +\infty$.\\[0.5ex]
    Then $T_{\mathcal P\exm}\, u_p$ defined by \eqref{transT-u_gene} tends to zero when $m$ tends to $+\infty$ uniformly with respect to $p \in \xN$.
\end{theorem}

\section*{Acknowledgements}
We are grateful to K\'evin Gantheil who contributed to the numerical results.

\bibliographystyle{abbrv}
\bibliography{bib}

\end{document}